\documentclass[11pt]{amsart}
\usepackage{amsmath,amssymb,latexsym,esint,cite,mathrsfs}
\usepackage{verbatim,wasysym}
\usepackage[left=2.6cm,right=2.6cm,top=2.9cm,bottom=2.9cm]{geometry}
\usepackage{xcolor}

\usepackage{graphicx}
\usepackage{subfigure}

\usepackage{graphicx,epic,eepic}

\allowdisplaybreaks

\begin{document}

\title[Hamiltonian systems]
{On a class of singular Hamiltonian Choquard-type elliptic systems with critical exponential growth}

\author[S. Deng]{Shengbing Deng}
\address{\noindent S. Deng-School of Mathematics and Statistics, Southwest University,
Chongqing 400715, People's Republic of China}\email{shbdeng@swu.edu.cn}

\author[J. Yu]{Junwei Yu}
\address{\noindent J. Yu-School of Mathematics and Statistics, Southwest University,
Chongqing 400715, People's Republic of China.}\email{JwYumaths@163.com}

\maketitle

\maketitle
\numberwithin{equation}{section}
\newtheorem{theorem}{Theorem}[section]
\newtheorem{lemma}[theorem]{Lemma}
\newtheorem{corollary}[theorem]{Corollary}
\newtheorem{definition}[theorem]{Definition}
\newtheorem{proposition}[theorem]{Proposition}
\newtheorem{remark}[theorem]{Remark}
\allowdisplaybreaks

\maketitle

\noindent {\bf Abstract}: In this paper, we study the following Hamiltonian Choquard-type elliptic systems involving singular weights
\begin{eqnarray*}
    \begin{aligned}\displaystyle
    \left\{ \arraycolsep=1.5pt
       \begin{array}{ll}
       -\Delta u + V(x)u = \Big(I_{\mu_{1}}\ast \frac{G(v)}{|x|^{\alpha}}\Big)\frac{g(v)}{|x|^{\alpha}} \ \ \ & \mbox{in} \ \mathbb{R}^{2},\\[2mm]
       -\Delta v + V(x)v = \Big(I_{\mu_{2}}\ast \frac{F(u)}{|x|^{\beta}}\Big)\frac{f(u)}{|x|^{\beta}} \ \ \ & \mbox{in} \ \mathbb{R}^{2},
        \end{array}
    \right.
    \end{aligned}
\end{eqnarray*}
where $\mu_{1},\mu_{2}\in(0,2)$,   $0<\alpha \leq \frac{\mu_{1}}{2}$, $0<\beta \leq \frac{\mu_{2}}{2}$, $V(x)$ is a continuous positive potential, $I_{\mu_{1}}$ and $I_{\mu_{2}}$ denote the Riesz potential, $\ast$ indicates the convolution operator, $F(s),G(s)$ are the primitive of $f(s),g(s)$ with $f(s),g(s)$ have exponential growth in $\mathbb{R}^{2}$. Using the linking theorem and variational methods, we establish the existence of solutions to the above problem.

\vspace{3mm} \noindent {\bf Keywords}: Hamiltonian system; Choquard nonlinearity; Critical exponential growth; Trudinger-Moser inequality; Singular weights.

\vspace{3mm}

\maketitle

\section{{\bfseries Introduction}}
This paper deals with the following class of Hamiltonian elliptic systems of the Choquard type involving singular weights
\begin{equation}\label{a}
    \begin{aligned}\displaystyle
    \left\{ \arraycolsep=1.5pt
       \begin{array}{ll}
       -\Delta u + V(x)u = \Big(I_{\mu_{1}}\ast \frac{G(v)}{|x|^{\alpha}}\Big)\frac{g(v)}{|x|^{\alpha}} \ \ \ & \mbox{in} \ \mathbb{R}^{2},\\[2mm]
       -\Delta v + V(x)v = \Big(I_{\mu_{2}}\ast \frac{F(u)}{|x|^{\beta}}\Big)\frac{f(u)}{|x|^{\beta}} \ \ \ & \mbox{in} \ \mathbb{R}^{2},
        \end{array}
    \right.
    \end{aligned}
\end{equation}
where $\mu_{1},\mu_{2}\in(0,2)$,   $0<\alpha \leq \frac{\mu_{1}}{2}$, $0<\beta \leq \frac{\mu_{2}}{2}$, $I_{\mu_{1}}$ and $I_{\mu_{2}}$ denote the Riesz potential defined by
\begin{eqnarray*}
    \begin{aligned}\displaystyle
    I_{\mu_{i}}(x)=\frac{\Gamma(\frac{2-\mu_{i}}{2})}{\Gamma(\frac{\mu_{i}}{2})2^{\mu_{i}}\pi|x|^{2-\mu_{i}}}:=\frac{A_{\mu_{i}}}{|x|^{2-\mu_{i}}}, \ \ x\in \mathbb{R}^{2}\backslash\{0\}, \ i=1,2,
    \end{aligned}
\end{eqnarray*}
where $\Gamma$ represents the gamma function, $\ast$ indicates the convolution operator, $F(t)=\int^{t}_{0}f(\tau)d\tau$ and $G(t)=\int^{t}_{0}g(\tau)d\tau$. The potential $V$ and the nonlinearities $f,g$ satisfy some suitable conditions that will be specified later.

We say that a function $h(s)$ has subcritical growth if for all $\gamma>0$ we have
\begin{equation}\label{ab}
    \begin{aligned}\displaystyle
    \lim\limits_{|s|\rightarrow +\infty}\frac{|h(s)|}{e^{\gamma s^{2}}}=0
    \end{aligned}
\end{equation}
and $h(s)$ has critical exponential growth if there exists $\gamma_{0}>0$ such that
\begin{equation}\label{ac}
    \begin{aligned}\displaystyle
    \lim\limits_{|s|\rightarrow +\infty}\frac{|h(s)|}{e^{\gamma s^{2}}}=0,\ \ \forall \ \gamma>\gamma_{0},\ \ and \ \ \lim\limits_{|s|\rightarrow +\infty}\frac{|h(s)|}{e^{\gamma s^{2}}}=+ \infty,\ \ \forall \ \gamma<\gamma_{0}.
    \end{aligned}
\end{equation}
This notion of criticality is motivated by Trudinger-Moser inequality.

In recent years, there are many works dedicated to study the following non-local elliptic equation
\begin{equation}\label{b}
    \begin{aligned}\displaystyle
    -\Delta u +u =(I_{\mu} \ast F(u))f(u) \ \ \mbox{in} \ \mathbb{R}^{2}.
    \end{aligned}
\end{equation}
This nonlocal equation arises in many interesting physical situations in quantum theory and plays an important role in describing the finite-range many-body interactions. Equation $(\ref{b})$ is well known in the literature as the Choquard equation and was first introduced in the pioneering work of Fr$\ddot{o}$hlich \cite{FR} for the modeling of quantum polaron. We mention that a great attention has been focused on the study of problem involving exponential growth nonlinearity. In \cite{BATTA}, Battaglia and Van Schaftingen proved the existence of a nontrivial ground state solution for equation $(\ref{b})$ when $F$ has exponential growth. Using the Moser type functions, Alves et al. \cite{ACTT} obtained the existence and concentration results for a class of nonlinear Choquard equations in the plane with exponential critical nonlinearity. Qin and Tang \cite{QT} improved and extended the related results to the strongly indefinite problems. For more classical results regarding Choquard equations, we refer to \cite{MVb} for a good survey.

On the other hand, we mention that the study of our problem is based on some interesting results of the following Hamiltonian elliptic system:
\begin{equation}\label{c}
    \begin{aligned}\displaystyle
    \left\{ \arraycolsep=1.5pt
       \begin{array}{ll}
       -\Delta u + V(x)u = f_{1}(x,v) \ \ \ & \mbox{in} \ \Omega,\\[2mm]
       -\Delta v + V(x)v = f_{2}(x,u) \ \ \ & \mbox{in} \ \Omega,
        \end{array}
    \right.
    \end{aligned}
\end{equation}
where $\Omega\subset \mathbb{R}^{2}$ is a bound domain. For $V=0$, de Figueiredo et al. \cite{FdR} studied the existence of solution for $(\ref{c})$ with Ambrosetti-Rabinowitz condition. Lam and Lu \cite{LLU} considered the existence of nontrivial nonnegative solutions for $(\ref{c})$ with subcritical and critical exponential growth without the $(AR)$ condition. When $V=1$ and $\Omega=\mathbb{R}^{2}$, the existence of ground state solutions of $(\ref{c})$ was introduced by de Figueiredo et al. \cite{FJZ}. For more classical results regarding Hamiltonian system, the reader may refer \cite{Bo} to the good survey.

When $f_{1}(x,v)=\frac{g(v)}{|x|^{\alpha}}$ and $f_{2}(x,u)=\frac{f(u)}{|x|^{\beta}}$, system $(\ref{c})$ becomes the singular Hamiltonian system:
\begin{equation}\label{d}
    \begin{aligned}\displaystyle
    \left\{ \arraycolsep=1.5pt
       \begin{array}{ll}
       -\Delta u + V(x)u = \frac{g(v)}{|x|^{\alpha}} \ \ v>0 \ \ \ & \mbox{in} \ \mathbb{R}^{2},\\[2mm]
       -\Delta v + V(x)v = \frac{f(u)}{|x|^{\beta}} \ \ u>0 \ \ \ & \mbox{in} \ \mathbb{R}^{2},
        \end{array}
    \right.
    \end{aligned}
\end{equation}
where $\alpha,\beta\in [0,2)$, $V: \mathbb{R}^{2}\rightarrow \mathbb{R}$ is continuous, positive and $ [V(x)]^{-1}\in L^{1}(\mathbb{R}^{2}) $. Souza \cite{Souza} established the existence of nontrivial solutions for the system $(\ref{d})$. More precisely, Souza showed the existence results of system $(\ref{d})$ with subcritical nonlinearities or critical nonlinearities$(\alpha=\beta)$ in the Trudinger-Moser sense. For other singular Hamiltonian systems, we can refer to \cite{Alb,Ben,Milan}.

When the nonlinear term is Choquard type, Maia and Miyagaki \cite{MMA} studied system $(\ref{c})$ of the type
\begin{equation}\label{e}
    \begin{aligned}\displaystyle
    \left\{ \arraycolsep=1.5pt
       \begin{array}{ll}
       -\Delta u + V(x)u = \Big(I_{\mu}\ast G(v)\Big)g(v) \ \ \ & \mbox{in} \ \mathbb{R}^{2},\\[2mm]
       -\Delta v + V(x)v = \Big(I_{\mu}\ast F(v)\Big)f(v) \ \ \ & \mbox{in} \ \mathbb{R}^{2},
        \end{array}
    \right.
    \end{aligned}
\end{equation}
where $I_{\mu}=\frac{1}{|x|^{\mu}}$ denotes the Riesz potential, the potential $V$ and the nonlinearities satisfy:

$(H_{0})$ $f,g: \mathbb{R} \rightarrow [0,+\infty)$ are continuous functions, and both have critical exponential growth;

$(H_{1})$ $\lim\limits_{s\rightarrow 0^{+}}f(s)s^{-1}=\lim\limits_{s\rightarrow 0^{+}}g(s)s^{-1}=0$;

$(H_{2})$ there exists $\theta > 2$, such that $0<\theta F(s)\leq sf(s)$ and $ 0< \theta G(s) \leq sg(s)$ for all $s>0$;

$(H_{3})$ there exist constants $s_{0},M_{0}>0$, such that $0<F(s)\leq M_{0}f(s)$ and $0< G(s) \leq M_{0}g(s)$ for all $s>s_{0}$;

$(H_{4})$ $f$ and $g$ are locally bounded;

$(H_{5})$ there exist constants $p>2$ and $C_{p}$ such that for all $s\geq0$,
\begin{eqnarray*}
    \begin{aligned}\displaystyle
    F(s),G(s)\geq2(\frac{C_{p}}{p})^{\frac{1}{2}}s^{p},
    \end{aligned}
\end{eqnarray*}
where
\begin{eqnarray*}
    \begin{aligned}\displaystyle
 C_{p}>(\frac{\overline{C}\alpha_{0}(p-1)}{4\pi p})^{\frac{p-1}{p}}S_{p}^{2}
    \end{aligned}
\end{eqnarray*}
and
\begin{eqnarray*}
    \begin{aligned}\displaystyle
    S_{p}=\inf\limits_{u\in H^{1}_{V}(\mathbb{R}^{2}),u\neq 0}\frac{\int_{\mathbb{R}^{2}}(|\nabla u|^{2}+V(x)u^{2})}{\Big(\int_{\mathbb{R}^{2}}\Big[\frac{1}{|x|^{\mu}}\ast u^{p}\Big]u^{p}\Big)^{\frac{1}{p}}},
    \end{aligned}
\end{eqnarray*}
where $\overline{C}>\frac{4}{4-\mu}$. By $(H_{5})$, Maia and Miyagaki \cite{MMA} obtained an upper bound
for the minimax level. Using approximation methods and linking arguments, they showed the existence of solutions for $(\ref{e})$ with critical exponential growth. Moreover, Maia and Miyagaki \cite{MMB} investigated the existence and nonexistence results of Hamiltonian Choquard-type elliptic systems with lower critical growth.

Recently, Tang et al. \cite{QTZA} obtained the existence of ground states for $(\ref{c})$ by the non-Nehari manifold method. They developed a direct method to deal with the difficulties aroused by the strongly indefinite features and the critical exponential growth.

The purpose of this work is to consider the existence of solutions for the singular Hamiltonian Choquard-type elliptic systems (\ref{a}).
Assume $V$ is a continuous function satisfying:

$(V_1)$ $V(x)\geq \inf_{\mathbb{R}^{2}}V(x):=V_{0}>0$;

$(V_2)$ The function $[V(x)]^{-1}$ belongs to $L^{1}(\mathbb{R}^{2})$.

Moreover, we assume that $f(s)=g(s)=0$ for $t\leq0$, and the nonlinearities $f,g$ satisfy the following conditions:

$(H_0^{\prime})$ $f,g: \mathbb{R} \rightarrow [0,+\infty)$ are continuous functions;

$(H_1^{\prime})$ $\lim\limits_{s\rightarrow 0^{+}}f(s)s^{-1}=\lim\limits_{s\rightarrow 0^{+}}g(s)s^{-1}=0$;

$(H_2^{\prime})$ there exists $\theta > 2$ such that for all $s>0$,
    \begin{eqnarray*}
    \begin{aligned}\displaystyle
    0<\theta F(s)\leq sf(s) \ \ and \ \ 0< \theta G(s) \leq sg(s);
    \end{aligned}
    \end{eqnarray*}

$(H_3^{\prime})$ there exist constants $s_{0},M_{0}>0$ such that for all $s>s_{0}$,
    \begin{eqnarray*}
    \begin{aligned}\displaystyle
    0<F(s)\leq M_{0}f(s) \ \ and \ \ 0< G(s) \leq M_{0}g(s);
    \end{aligned}
    \end{eqnarray*}

$(H_4^{\prime})$ $\liminf\limits_{t\rightarrow +\infty}\frac{tF(t)}{e^{\gamma_{0}t^{2}}}\geq \kappa $ and $\liminf\limits_{t\rightarrow +\infty}\frac{tG(t)}{e^{\gamma_{0}t^{2}}}\geq \kappa $ with $\kappa > \inf\limits_{\rho>0} \sqrt{\frac{(2-2\alpha+\mu_{1})\mu_{1}(1+\mu_{1})(2+\mu_{1})} {8\pi A_{\mu_{1}}\rho^{2-2\alpha+\mu_{1}}\gamma_{0}^{2}}e^{\frac{V_{\rho}\rho^{2}(2-2\alpha+\mu_{1})}{4}-1}}$, where $V_{\rho}:=\sup\limits_{|x|\leq\rho}V(x)$.

Our results for the Hamiltonian elliptic system $(\ref{a})$ read as follows.
\begin{theorem}\label{THa}
Suppose that $g$ has subcritical growth, $f$ has subcritical or critical growth, $(V_{1}),(V_{2})$ and $(H_{0}^{\prime})-(H_{3}^{\prime})$ hold, then system $(\ref{a})$ has a nontrivial positive weak solution.
\end{theorem}

\begin{theorem}\label{THb}
Suppose that $\alpha=\beta$, $\mu_{1}=\mu_{2}$, $f,g$ both have critical exponential growth, $(V_{1}),(V_{2})$ and $(H_{0}^{\prime})-(H_{4}^{\prime})$ hold, then system $(\ref{a})$ has a nontrivial positive weak solution.
\end{theorem}

\begin{remark}\rm
We notice that, there are some difficulties to study the existence of solutions for system $(\ref{a})$. The first one is that the energy functional associated with system $(\ref{a})$ is strongly indefinite and there is a lack of the compactness for the whole space case. Hence, we use a Galerkin approximation procedure to overcome the lack of compactness. The second one is the nonlinearities with critical exponent and singular weights, which also lead to the lack of compactness. Different from \cite{MMA}, we restore compactness by the properties of the exponential critical nonlinear term at infinity. Therefore, we introduce the following $(H_{4}^{\prime})$ condition instead of $(H_{5})$.Then using Trudinger-Moser inequality inequality and Doubly Weighted Hardy-Littlewood-Sobolev inequality, introduced  in Section $\ref{PFS}$ , we address this issue.
\end{remark}
The paper is organized as follows: In Section $\ref{PFS}$, the variational setting and some preliminary results are presented. In Section $\ref{STLEMMA}$, we introduce the geometry structure and give some properties related to system $(\ref{a})$. In Section $\ref{FIn}$, we give a Galerkin
approximation procedure of system $(\ref{a})$. In Section $\ref{Subp}$, using an approaching argument and variational methods we complete the proof of the subcritical case. Section $\ref{TECL}$ is devoted to give an estimate for the minimax level of the critical case. Finally in Section $\ref{PROOF}$, we complete the proof of critical case.

\section{{\bfseries Preliminaries and functional setting}}\label{PFS}
In this section, we give some preliminary results and outline the variational framework for $(\ref{a})$.
\begin{proposition}\label{PRa}
\cite{Ca}.$i)$ If $\gamma>0$ and $u\in H^{1}(\mathbb{R}^{2})$, then
\begin{eqnarray*}
    \begin{aligned}\displaystyle
    \int_{\mathbb{R}^{2}}\Big(e^{\gamma u^{2}-1}\Big)dx<\infty;
    \end{aligned}
\end{eqnarray*}
$ii)$ if $u\in H^{1}(\mathbb{R}^{2})$, $\|\nabla u\|^{2}_{2}\leq1$, $\|u\|_{2}\leq M<\infty$, and $\gamma<4\pi$, then there exists a constant $\mathcal{C}(M,\gamma)$, which depends only on $M$ and $\gamma$, such that
\begin{eqnarray*}
    \begin{aligned}\displaystyle
    \int_{\mathbb{R}^{2}}\Big(e^{\gamma u^{2}-1}\Big)dx<\mathcal{C}(M,\gamma).
    \end{aligned}
\end{eqnarray*}
\end{proposition}

\begin{proposition}\label{PRb}
\cite{AdY}. If $\gamma>0$, $0\leq \alpha<2$ and $u\in H^{1}(\mathbb{R}^{2})$, then
\begin{eqnarray*}
    \begin{aligned}\displaystyle
    \int\limits_{\mathbb{R}^{2}}\frac{e^{\gamma|u|^{2}}-1}{|x|^{\alpha}}dx<\infty.
    \end{aligned}
\end{eqnarray*}
Moreover, if $\|\nabla u\|^{2}_{2}\leq1$, $\|u\|_{2}\leq M<\infty$ and $\gamma/4\pi+\alpha/2<1$. Then exists a constant $C(M,\gamma)$ such that
\begin{eqnarray*}
    \begin{aligned}\displaystyle
\sup\limits_{\|\nabla u\|_{2}\leq1}\int\limits_{\mathbb{R}^{2}}\frac{e^{\gamma|u|^{2}}-1}{|x|^{\alpha}}dx<C(M,\gamma)
    \end{aligned}
\end{eqnarray*}
and the above inequality does not holds if $\gamma/4\pi+\alpha/2>1$.
\end{proposition}

We recall that the embedding
\begin{eqnarray*}
    \begin{aligned}\displaystyle
    H^{1}(\mathbb{R}^{2})\ni u \mapsto \frac{e^{|u|^{\zeta}}}{|x|^{\alpha}}\in L^{1}(\mathbb{R}^{2})
    \end{aligned}
\end{eqnarray*}
is compact for all $\zeta\in \ [1,2)$ and continuous for $\zeta = 2$.\\

Now we state the doubly-weighted Hardy-Littlewood-Sobolev inequality proved in \cite{SW}.

\begin{proposition}\label{PRc}(Doubly Weighted Hardy-Littlewood-Sobolev inequality). Let $t,r>1$, $0<\mu<2$, $\alpha+\beta\geq0$ and $\alpha+\beta+\mu\leq 2$ with $\frac{1}{t}+\frac{\mu+\alpha+\beta}{2}+\frac{1}{r}=2$, $\alpha<\frac{2}{t^{\prime}}$, $\beta<\frac{2}{r^{\prime}}$ $f\in L^{t}(\mathbb{R}^{2})$ and $h\in L^{r}(\mathbb{R}^{2})$, where $t^{\prime}$ and $r^{\prime}$ denote the Holder conjugate of $t$ and $r$ respectively. Then there exists a sharp constant $C(\alpha,\beta,t,\mu,r)$, independent of $f,h$, such that
\begin{eqnarray*}
    \begin{aligned}\displaystyle
    \int_{\mathbb{R}^{2}}\int_{\mathbb{R}^{2}}\frac{f(x)h(y)}{|x-y|^{\mu}|y|^{\alpha}|x|^{\beta}}dxdy \leq C(\alpha,\beta,t,\mu,r)\|f\|_{L^{t}(\mathbb{R}^{2})}\|h\|_{L^{r}(\mathbb{R}^{2})}.
    \end{aligned}
\end{eqnarray*}
\end{proposition}

Applying the Doubly Weighted Hardy-Littlewood-Sobolev inequality, we know
\begin{eqnarray*}
    \begin{aligned}\displaystyle
    \int_{\mathbb{R}^{2}}\Big(I_{\mu_{2}}\ast \frac{F(u)}{|x|^{\beta}}\Big)\frac{F(u)}{|x|^{\beta}}dx, \ \ \int_{\mathbb{R}^{2}}\Big(I_{\mu_{1}}\ast \frac{G(u)}{|x|^{\alpha}}\Big)\frac{G(u)}{|x|^{\alpha}}dx
    \end{aligned}
\end{eqnarray*}
are well defined if $F(u),G(u)\in L^{t}(\mathbb{R}^{2})$ for $t>1$ given by
\begin{eqnarray*}
    \begin{aligned}\displaystyle
    \frac{2}{t}+\frac{2-\mu_{2}+2\beta}{2}=2, \ \ \frac{2}{t}+\frac{2-\mu_{1}+2\alpha}{2}=2.
    \end{aligned}
\end{eqnarray*}
This means that we must require
\begin{eqnarray*}
    \begin{aligned}\displaystyle
    F(u)\in L^{\frac{4}{2+\mu_{2}-2\beta}}(\mathbb{R}^{2}), \ \ G(u)\in L^{\frac{4}{2+\mu_{1}-2\alpha}}(\mathbb{R}^{2}).
    \end{aligned}
\end{eqnarray*}

In order to apply variational methods, we consider the following subspace of $H^{1}(\mathbb{R}^{2})$
\begin{eqnarray*}
    \begin{aligned}\displaystyle
    H^{1}_{V}(\mathbb{R}^{2})=\{u\in H^{1}(\mathbb{R}^{2}): \int_{\mathbb{R}^{2}}V(x)u^{2}dx<\infty\},
    \end{aligned}
\end{eqnarray*}
which is a Hilbert space when endowed with the inner product
\begin{eqnarray*}
    \begin{aligned}\displaystyle
     \langle u,v \rangle=\int_{\mathbb{R}^{2}}[\nabla u\nabla v + V(x)uv]dx, \ \ u,v\in H^{1}_{V}(\mathbb{R}^{2})
    \end{aligned}
\end{eqnarray*}
to which corresponds the norm
\begin{eqnarray*}
    \begin{aligned}\displaystyle
    \|u\|=\Big\{\int_{\mathbb{R}^{2}}[|\nabla u|^{2} + V(x)|u|^{2}]dx\Big\}^{\frac{1}{2}}.
    \end{aligned}
\end{eqnarray*}

Assumptions $(V_{1})$ and $(V_{2})$ imply that the embedding $H^{1}_{V}(\mathbb{R}^{2})\hookrightarrow L^{q}(\mathbb{R}^{2})$ for all $1\leq q <\infty$ are continuous, and the embedding
\begin{equation}\label{embed}
    \begin{aligned}\displaystyle
 H^{1}_{V}(\mathbb{R}^{2})\hookrightarrow L^{q}(\mathbb{R}^{2})
    \end{aligned}
\end{equation}
are compact for all $1\leq q <\infty$(see\cite{Souza}).
Set $E:=H^{1}_{V}(\mathbb{R}^{2})\times H^{1}_{V}(\mathbb{R}^{2})$. Then E endowed with the inner product
\begin{eqnarray*}
    \begin{aligned}\displaystyle
     \langle (u,v),(\varphi,\psi) \rangle_{E}=\int_{\mathbb{R}^{2}}[\nabla u\nabla \varphi +\nabla v\nabla \psi + V(x)u\varphi +V(x)v\psi]dx
    \end{aligned}
\end{eqnarray*}
for all $(u,v),(\varphi,\psi)\in E$, to which corresponds the norm $\|(u,v)\|_{E}=\langle (u,v), (u,v)\rangle_{E}^{\frac{1}{2}}$.

Then the energy functional associated with $(\ref{a})$ is given by $\Phi:E\rightarrow R$, where
\begin{eqnarray*}
    \begin{aligned}\displaystyle
    \Phi(u,v)=\langle u,v\rangle-\frac{1}{2}\int_{\mathbb{R}^{2}}\Big(I_{\mu_{2}}\ast \frac{F(u)}{|x|^{\beta}}\Big)\frac{F(u)}{|x|^{\beta}}dx-\frac{1}{2}\int_{\mathbb{R}^{2}}\Big(I_{\mu_{1}}\ast \frac{G(v)}{|x|^{\alpha}}\Big)\frac{G(v)}{|x|^{\alpha}}dx.
    \end{aligned}
\end{eqnarray*}
Using standard arguments it is possible to verify that $\Phi$ is well defined and is of class $C^{1}$ with
\begin{eqnarray*}
    \begin{aligned}\displaystyle
    \Phi^{\prime}(u,v)(\varphi,\psi)&=\int_{\mathbb{R}^{2}}[\nabla u\nabla \psi +\nabla v\nabla \varphi +V(x)u\psi+V(x)v\varphi]\\
    &-\int_{\mathbb{R}^{2}}\Big(I_{\mu_{2}}\ast \frac{F(u)}{|x|^{\beta}}\Big)\frac{f(u)\varphi}{|x|^{\beta}}dx-\int_{\mathbb{R}^{2}}\Big(I_{\mu_{1}}\ast \frac{G(v)}{|x|^{\alpha}}\Big)\frac{g(v)\psi}{|x|^{\alpha}}dx \ \ \mbox{for} \ \ \mbox{all} \ \ (\varphi,\psi)\in E.
    \end{aligned}
\end{eqnarray*}

Let us define the following subspaces of $E$,
\begin{eqnarray*}
    \begin{aligned}\displaystyle
    E^{+}=\{(u,u)\in E\}\ \ and\ \  E^{-}=\{(v,-v)\in E\}
    \end{aligned}
\end{eqnarray*}
with
\begin{eqnarray*}
    \begin{aligned}\displaystyle
    (u,v)=\frac{1}{2}(u+v,u+v)+\frac{1}{2}(u-v,v-u),
    \end{aligned}
\end{eqnarray*}
then $E^{+}$ is orthogonal to $E^{-}$ with respect to the inner product $\langle \cdot,\cdot\rangle$. For every $z=(u,v)$, set
\begin{eqnarray*}
    \begin{aligned}\displaystyle
    z^{+}:=\Big(\frac{u+v}{2},\frac{u+v}{2}\Big) \ and \ z^{-}:=\Big(\frac{u-v}{2},\frac{v-u}{2}\Big).
    \end{aligned}
\end{eqnarray*}
Note that $z^{+}\in E^{+}$ and $z^{-}\in E^{-}$, then we have $E=E^{+}\bigoplus E^{-}$.

From $(H_{0}^{\prime}),(H_{1}^{\prime})$, we have the following immediate result: for all $\gamma>\gamma_{0}$ and $q\geq1$, for any given $\varepsilon >0$, there exist constants $b_{1},b_{2}>0$ such that
\begin{equation}\label{fnon}
    \begin{aligned}\displaystyle
    f(s),g(s)\leq\varepsilon|s|+b_{1}|s|^{q-1}(e^{\gamma s^{2}}-1),\ \ \mbox{for} \ \mbox{all} \ s\in \mathbb{R},
    \end{aligned}
\end{equation}
and
\begin{equation}\label{fnonb}
    \begin{aligned}\displaystyle
    F(s),G(s)\leq\frac{\varepsilon}{2}|s|^{2}+b_{2}|s|^{q}(e^{\gamma s^{2}}-1),\ \ \mbox{for} \ \mbox{all} \ s\in \mathbb{R}.
    \end{aligned}
\end{equation}

\begin{lemma}\label{LEbd}(Cauchy-Schwarz type inequality)\cite{ack}. For $g,h\in L^{1}_{loc}(\mathbb{R}^{2})$, there holds
\begin{eqnarray*}
    \begin{aligned}\displaystyle
    \int_{\mathbb{R}^{2}}(I_{\mu}\ast|g|)|h|dx\leq \Big[\int_{\mathbb{R}^{2}}(I_{\mu}\ast|g|)|g|dx\Big]^{\frac{1}{2}} \Big[\int_{\mathbb{R}^{2}}(I_{\mu}\ast|h|)|h|dx\Big]^{\frac{1}{2}}.
    \end{aligned}
\end{eqnarray*}
\end{lemma}

\section{{\bfseries On the Palais-Smale sequence}}\label{STLEMMA}
\begin{lemma}\label{LEca}The following inequality holds:
\begin{eqnarray*}
    \begin{aligned}\displaystyle
    st\leq
    \left\{ \arraycolsep=1.5pt
       \begin{array}{ll}
        (e^{t^{2}}-1)+s(\log s)^{\frac{1}{2}}, \ \ \ & for \ all \ t\geq0 \ and \ s\geq e^{\frac{1}{4}};\\[2mm]
        (e^{t^{2}}-1)+\frac{1}{2}s^{2} \ \ \ & for \ all \ t\geq0 \ and \ 0\leq s\leq e^{\frac{1}{4}}.\\[2mm]
        \end{array}
    \right.
    \end{aligned}
\end{eqnarray*}
\end{lemma}
\begin{proof}
See Lemma $2.4$ of \cite{FJZ}.
\end{proof}

\begin{lemma}\label{LEcb} Assume that $(H_{0}^{\prime})-(H_{2}^{\prime})$ hold. If $(u_{n},v_{n})\in E$ such that
\begin{equation}\label{PSa}
    \begin{aligned}\displaystyle
    \Phi(u_{n},v_{n})\rightarrow c \ \ and \ \ \Phi^{\prime}(u_{n},v_{n})(\varphi,\psi)\rightarrow 0 \ \ for \ all \ (\varphi,\psi)\in E,
    \end{aligned}
\end{equation}
then $(u_{n},v_{n})$ is bounded in $E$, and there exists a constant $C>0$ such that
\begin{equation}\label{PSb}
    \begin{aligned}\displaystyle
    \int_{\mathbb{R}^{2}}\Big(I_{\mu_{2}}\ast \frac{F(u_{n})}{|x|^{\beta}}\Big)\frac{f(u_{n})u_{n}}{|x|^{\beta}}dx\leq C, \quad \ \int_{\mathbb{R}^{2}}\Big(I_{\mu_{1}}\ast \frac{G(v_{n})}{|x|^{\alpha}}\Big)\frac{G(v_{n})v_{n}}{|x|^{\alpha}}dx\leq C,
    \end{aligned}
\end{equation}
\begin{equation}\label{PSc}
    \begin{aligned}\displaystyle
    \int_{\mathbb{R}^{2}}\Big(I_{\mu_{2}}\ast \frac{F(u_{n})}{|x|^{\beta}}\Big)\frac{F(u_{n})}{|x|^{\beta}}dx\leq C, \ \quad\ \  \int_{\mathbb{R}^{2}}\Big(I_{\mu_{1}}\ast \frac{G(v_{n})}{|x|^{\alpha}}\Big)\frac{G(v_{n})}{|x|^{\alpha}}dx\leq C.
    \end{aligned}
\end{equation}
Moreover, for every $\phi \in C^{\infty}_{0}(\mathbb{R}^{2})$,
\begin{equation}\label{PSd}
    \begin{aligned}\displaystyle
    &\lim\limits_{n\rightarrow\infty}\int_{\mathbb{R}^{2}}\Big(I_{\mu_{2}}\ast \frac{F(u_{n})}{|x|^{\beta}}\Big)\frac{f(u_{n})\phi}{|x|^{\beta}}dx=\int_{\mathbb{R}^{2}}\Big(I_{\mu_{2}}\ast \frac{F(u)}{|x|^{\beta}}\Big)\frac{f(u)\phi}{|x|^{\beta}}dx,\\
    &\lim\limits_{n\rightarrow\infty}\int_{\mathbb{R}^{2}}\Big(I_{\mu_{1}}\ast \frac{G(v_{n})}{|x|^{\alpha}}\Big)\frac{g(v_{n})\phi}{|x|^{\alpha}}dx=\int_{\mathbb{R}^{2}}\Big(I_{\mu_{1}}\ast \frac{G(v)}{|x|^{\alpha}}\Big)\frac{g(v)\phi}{|x|^{\alpha}}dx.
    \end{aligned}
\end{equation}
\end{lemma}

\begin{proof} Taking $(\varphi,\psi)=(u_{n},v_{n})$ in $(\ref{PSa})$, we have
\begin{eqnarray*}
    \begin{aligned}\displaystyle
    2 \| (u_{n},v_{n}) \|^{2}_{E} -\int_{\mathbb{R}^{2}}\Big(I_{\mu_{2}}\ast \frac{F(u_{n})}{|x|^{\beta}}\Big)\frac{f(u_{n})u_{n}}{|x|^{\beta}}dx- \int_{\mathbb{R}^{2}}\Big(I_{\mu_{1}}\ast \frac{G(v_{n})}{|x|^{\alpha}}\Big)\frac{g(v_{n})v_{n}}{|x|^{\alpha}}dx = \varepsilon_{n},
    \end{aligned}
\end{eqnarray*}
which together with $(\ref{PSa})$ and hypothesis $(H_{2}^{\prime})$, implies
\begin{eqnarray*}
    \begin{aligned}\displaystyle
    &\int_{\mathbb{R}^{2}}\Big(I_{\mu_{2}}\ast \frac{F(u_{n})}{|x|^{\beta}}\Big)\frac{f(u_{n})u_{n}}{|x|^{\beta}}dx+ \int_{\mathbb{R}^{2}}\Big(I_{\mu_{1}}\ast \frac{G(v_{n})}{|x|^{\alpha}}\Big)\frac{g(v_{n})v_{n}}{|x|^{\alpha}}dx\\
=& 2c +\varepsilon_{n} + \int_{\mathbb{R}^{2}}\Big(I_{\mu_{2}}\ast \frac{F(u_{n})}{|x|^{\beta}}\Big)\frac{F(u_{n})}{|x|^{\beta}}dx+ \int_{\mathbb{R}^{2}}\Big(I_{\mu_{1}}\ast \frac{G(v_{n})}{|x|^{\alpha}}\Big)\frac{G(v_{n})}{|x|^{\alpha}}dx\\
\leq& 2c +\varepsilon_{n} +\frac{1}{\theta} \int_{\mathbb{R}^{2}}\Big(I_{\mu_{2}}\ast \frac{F(u_{n})}{|x|^{\beta}}\Big)\frac{f(u_{n})u_{n}}{|x|^{\beta}}dx+ \frac{1}{\theta} \int_{\mathbb{R}^{2}}\Big(I_{\mu_{1}}\ast \frac{G(v_{n})}{|x|^{\alpha}}\Big)\frac{g(v_{n})v_{n}}{|x|^{\alpha}}dx.
    \end{aligned}
\end{eqnarray*}
Thus, we obtain
\begin{eqnarray*}
    \begin{aligned}\displaystyle
    \int_{\mathbb{R}^{2}}\Big(I_{\mu_{2}}\ast \frac{F(u_{n})}{|x|^{\beta}}\Big)\frac{f(u_{n})u_{n}}{|x|^{\beta}}dx+ \int_{\mathbb{R}^{2}}\Big(I_{\mu_{1}}\ast \frac{G(v_{n})}{|x|^{\alpha}}\Big)\frac{g(v_{n})v_{n}}{|x|^{\alpha}}dx \leq \frac{\theta}{\theta -1} (2c +\varepsilon_{n})
    \end{aligned}
\end{eqnarray*}
where $\varepsilon_{n}\rightarrow0$, as $n\rightarrow\infty$. Next taking $(\varphi,\psi)=(v_{n},0)$ and $(\varphi,\psi)=(0,u_{n})$ in $(\ref{PSa})$ we have
\begin{eqnarray*}
    \begin{aligned}\displaystyle
    \|v_{n}\|^{2}\leq \int_{\mathbb{R}^{2}}\Big(I_{\mu_{2}}\ast \frac{F(u_{n})}{|x|^{\beta}}\Big)\frac{f(u_{n})v_{n}}{|x|^{\beta}}dx,
    \end{aligned}
\end{eqnarray*}
and
\begin{eqnarray*}
    \begin{aligned}\displaystyle
    \|u_{n}\|^{2}\leq \int_{\mathbb{R}^{2}}\Big(I_{\mu_{1}}\ast \frac{G(v_{n})}{|x|^{\alpha}}\Big)\frac{g(v_{n})u_{n}}{|x|^{\alpha}}dx.
    \end{aligned}
\end{eqnarray*}

Setting $U_{n}=u_{n}/\|u_{n}\|$ and $V_{n}=v_{n}/\|v_{n}\|$, we conclude that
\begin{eqnarray*}
    \begin{aligned}\displaystyle
    \|v_{n}\| \leq  \int_{\mathbb{R}^{2}}\Big(I_{\mu_{2}}\ast \frac{F(u_{n})}{|x|^{\beta}}\Big)\frac{f(u_{n})V_{n}}{|x|^{\beta}}dx,
    \end{aligned}
\end{eqnarray*}
and
\begin{eqnarray*}
    \begin{aligned}\displaystyle
    \|u_{n}\|\leq  \int_{\mathbb{R}^{2}}\Big(I_{\mu_{1}}\ast \frac{G(v_{n})}{|x|^{\alpha}}\Big)\frac{g(v_{n})U_{n}}{|x|^{\alpha}}dx.
    \end{aligned}
\end{eqnarray*}

From the exponential growth on $f$ and $(H_{1}^{\prime})$, we have
\begin{equation}\label{PSe}
    \begin{aligned}\displaystyle
    f(s)\leq C_{1}e^{\gamma s^{2}}, \ \mbox{for} \ \mbox{all} \ s\geq 0.
    \end{aligned}
\end{equation}

Using Lemma $\ref{LEca}$ with $t=V_{n}$ and $s=f(u_{n})/C_{1}$, where $C_{1}$ is the constant appearing in $(\ref{PSe})$, and the Trudinger-Moser inequality, we obtain
\begin{eqnarray*}
    \begin{aligned}\displaystyle
    &C_{1} \int_{\mathbb{R}^{2}}\Big(I_{\mu_{2}}\ast \frac{F(u_{n})}{|x|^{\beta}}\Big)\frac{f(u_{n})}{C_{1}}\frac{V_{n}}{|x|^{\beta}}dx\leq C_{1} \int_{\mathbb{R}^{2}}\Big(I_{\mu_{2}}\ast \frac{F(u_{n})}{|x|^{\beta}}\Big)\frac{e^{V_{n}^{2}}-1}{|x|^{\beta}}dx\\
    &+ C_{1} \int_{\{x\in \mathbb{R}^{2}:f(u_{n})/C_{1}\geq e^{1/4}\}}\Big(I_{\mu_{2}}\ast \frac{F(u_{n})}{|x|^{\beta}}\Big)\frac{f(u_{n})}{C_{1}|x|^{\beta}}\Big[\log \frac {f (u_{n})}{C_{1}}\Big]^{2}dx\\
    &+ \frac{1}{2} \int_{\{x\in \mathbb{R}^{2}:f(u_{n})/C_{1}\leq e^{1/4}\}}\Big(I_{\mu_{2}}\ast \frac{F(u_{n})}{|x|^{\beta}}\Big)\frac{[f(u_{n})]^{2}}{C_{1}^{2}|x|^{\beta}}dx\\
    &:= I_{1}+ I_{2} +I_{3}.
    \end{aligned}
\end{eqnarray*}

First, by the Hardy-Littlewood-Sobolev inequality, $V_{n}^{2}< 4\pi$, and $F(u_{n})\in L^{\frac{4}{2+\mu_{2}-2\beta}}$, we obtain
\begin{eqnarray*}
    \begin{aligned}\displaystyle
    I_{1}\leq |F(u_{n})|_{\frac{4}{2+\mu_{2}-2\beta}}|e^{V_{n}^{2}}-1|_{\frac{4}{2+\mu_{2}-2\beta}}
    \leq C_{2}.
    \end{aligned}
\end{eqnarray*}

Next by $(\ref{PSe})$, we get
\begin{eqnarray*}
    \begin{aligned}\displaystyle
I_{2} \leq C_{3} \int_{\mathbb{R}^{2}}\Big(I_{\mu_{2}}\ast \frac{F(u_{n})}{|x|^{\beta}}\Big)\frac{f(u_{n})u_{n}}{|x|^{\beta}}dx.
    \end{aligned}
\end{eqnarray*}

Finally, using $(H_{1}^{\prime})$, we have
\begin{eqnarray*}
    \begin{aligned}\displaystyle
    [f(s)]^{2} \ \leq \ f(s)s,
    \end{aligned}
\end{eqnarray*}
where $s\in \{s\in \mathbb{R}:s\geq0 \ and \ f(s)/C_{1} \leq e^{1/4} \}$. Thus
\begin{eqnarray*}
    \begin{aligned}\displaystyle
    I_{3}\leq C_{4} \int_{\mathbb{R}^{2}}\Big(I_{\mu_{2}}\ast \frac{F(u_{n})}{|x|^{\beta}}\Big)\frac{f(u_{n})u_{n}}{|x|^{\beta}}dx.
    \end{aligned}
\end{eqnarray*}
Thus, we obtain that
\begin{eqnarray*}
    \begin{aligned}\displaystyle
    C_{1} \int_{\mathbb{R}^{2}}\Big(I_{\mu_{2}}\ast \frac{F(u_{n})}{|x|^{\beta}}\Big)\frac{f(u_{n})}{C_{1}}\frac{V_{n}}{|x|^{\beta}}dx\leq C_{2}+ C_{5} \int_{\mathbb{R}^{2}}\Big(I_{\mu_{2}}\ast \frac{F(u_{n})}{|x|^{\beta}}\Big)\frac{f(u_{n})u_{n}}{|x|^{\beta}}dx,
    \end{aligned}
\end{eqnarray*}
which implies that
\begin{equation}\label{PSf}
    \begin{aligned}\displaystyle
    \|v_{n}\|\leq C_{2} + C \int_{\mathbb{R}^{2}}\Big(I_{\mu_{2}}\ast \frac{F(u_{n})}{|x|^{\beta}}\Big)\frac{f(u_{n})u_{n}}{|x|^{\beta}}dx.
    \end{aligned}
\end{equation}
Similarly, we have
\begin{equation}\label{PSg}
    \begin{aligned}\displaystyle
    \|u_{n}\|\leq C_{2} + C \int_{\mathbb{R}^{2}}\Big(I_{\mu_{1}}\ast \frac{G(v_{n})}{|x|^{\alpha}}\Big)\frac{g(v_{n})v_{n}}{|x|^{\alpha}}dx.
    \end{aligned}
\end{equation}

Now joining the estimates $(\ref{PSf})$ and $(\ref{PSg})$ we finally obtain
\begin{eqnarray*}
    \begin{aligned}\displaystyle
    \|u_{n}\|+\|v_{n}\|\leq\frac{\theta}{\theta-1}(2c+\varepsilon_{n}),
    \end{aligned}
\end{eqnarray*}
which implies that the boundedness of $(u_{n},v_{n})$. From this estimate, we obtain $(\ref{PSb})$ and $(\ref{PSc})$. Then by Lemma $2.1$ from \cite{DDRU}, we get $(\ref{PSd})$.
\end{proof}

\begin{lemma}\label{LEcc}Suppose that $(H_{0}^{\prime})-(H_{3}^{\prime})$ hold. If $(u_{n},v_{n})$ is a sequence such that $\Phi(u_{n})\rightarrow c$, $\Phi^{\prime}(u_{n},v_{n})\rightarrow 0$, and $(u_{0},v_{0})$ is its weak limit in E, then
\begin{eqnarray*}
    \begin{aligned}\displaystyle
    \int_{\mathbb{R}^{2}}\Big(I_{\mu_{2}}\ast \frac{F(u_{n})}{|x|^{\beta}}\Big)\frac{F(u_{n})}{|x|^{\beta}}dx \rightarrow \int_{\mathbb{R}^{2}}\Big(I_{\mu_{2}}\ast \frac{F(u_{0})}{|x|^{\beta}}\Big)\frac{F(u_{0})}{|x|^{\beta}}dx,
    \end{aligned}
\end{eqnarray*}
and
\begin{eqnarray*}
    \begin{aligned}\displaystyle
    \int_{\mathbb{R}^{2}}\Big(I_{\mu_{1}}\ast \frac{G(v_{n})}{|x|^{\alpha}}\Big)\frac{G(v_{n})}{|x|^{\alpha}}dx \rightarrow \int_{\mathbb{R}^{2}}\Big(I_{\mu_{1}}\ast \frac{G(v_{0})}{|x|^{\alpha}}\Big)\frac{G(v_{0})}{|x|^{\alpha}}dx.
    \end{aligned}
\end{eqnarray*}
\end{lemma}

\begin{proof}
First, for any $R>0$, we have,
\begin{eqnarray*}
    \begin{aligned}\displaystyle
    \int_{B_{R}(0)}\Big(I_{\mu_{2}}\ast \frac{F(u_{n})}{|x|^{\beta}}\Big)\frac{f(u_{n})}{|x|^{\beta}}dx \rightarrow \int_{B_{R}(0)}\Big(I_{\mu_{2}}\ast \frac{F(u_{0})}{|x|^{\beta}}\Big)\frac{f(u_{0})}{|x|^{\beta}}dx.
    \end{aligned}
\end{eqnarray*}
Then by $(H_{3}^{\prime})$ and the Lebesgue dominated convergence theorem, we conclude that
\begin{eqnarray*}
    \begin{aligned}\displaystyle
\lim\limits_{n\rightarrow\infty}\int_{B_{R}(0)}\Big(I_{\mu_{2}}\ast \frac{F(u_{n})}{|x|^{\beta}}\Big)\frac{F(u_{n})}{|x|^{\beta}}dx = \int_{B_{R}(0)}\Big(I_{\mu_{2}}\ast \frac{F(u_{0})}{|x|^{\beta}}\Big)\frac{F(u_{0})}{|x|^{\beta}}dx.
    \end{aligned}
\end{eqnarray*}

Now we show that given $\delta>0$, there exists $R>0$ such that
\begin{eqnarray*}
    \begin{aligned}\displaystyle
    \int_{\mathbb{R}^{2}\backslash B_{R}(0)}\Big(I_{\mu_{2}}\ast \frac{F(u_{n})}{|x|^{\beta}}\Big)\frac{F(u_{n})}{|x|^{\beta}}dx<\delta \ and \ \int_{\mathbb{R}^{2}\backslash B_{R}(0)}\Big(I_{\mu_{2}}\ast \frac{F(u_{0})}{|x|^{\beta}}\Big)\frac{F(u_{0})}{|x|^{\beta}}dx<\delta.
    \end{aligned}
\end{eqnarray*}

To prove this, choosing $K>0$, we obtain that
\begin{eqnarray*}
    \begin{aligned}\displaystyle
    &\int_{\{x\in\mathbb{R}^{2}\backslash B_{R},|u_{n}|>K\}}\Big(I_{\mu_{2}}\ast \frac{F(u_{n})}{|x|^{\beta}}\Big)\frac{F(u_{n})}{|x|^{\beta}}dx\\
\leq &M\int_{\{x\in\mathbb{R}^{2}\backslash B_{R},|u_{n}|>K\}}\Big(I_{\mu_{2}}\ast \frac{F(u_{n})}{|x|^{\beta}}\Big)\frac{f(u_{n})}{|x|^{\beta}}dx\\
\leq &\frac{M}{K}\int_{\{x\in\mathbb{R}^{2}\backslash B_{R},|u_{n}|>K\}}\Big(I_{\mu_{2}}\ast \frac{F(u_{n})}{|x|^{\beta}}\Big)\frac{f(u_{n})u_{n}}{|x|^{\beta}}dx\leq \frac{MC}{K}.
    \end{aligned}
\end{eqnarray*}
Taking $K$ large enough, we get
\begin{eqnarray*}
    \begin{aligned}\displaystyle
    \int_{\{x\in\mathbb{R}^{2}\backslash B_{R},|u_{n}|>K\}}\Big(I_{\mu_{2}}\ast \frac{F(u_{n})}{|x|^{\beta}}\Big)\frac{F(u_{n})}{|x|^{\beta}}dx <\frac{\delta}{2}.
    \end{aligned}
\end{eqnarray*}

Now, by $(H_{1}^{\prime})$, we have
\begin{eqnarray*}
    \begin{aligned}\displaystyle
    |F(s)|\leq C(\gamma,K)s^{2} \ for \ all \ s\in[-K,K].
    \end{aligned}
\end{eqnarray*}
Then, we obtain
\begin{eqnarray*}
    \begin{aligned}\displaystyle
    &\int_{\{x\in\mathbb{R}^{2}\backslash B_{R},|u_{n}|\leq K\}}\Big(I_{\mu_{2}}\ast \frac{F(u_{n})}{|x|^{\beta}}\Big)\frac{F(u_{n})}{|x|^{\beta}}dx\\
\leq& C(\gamma,K)\int_{\{x\in\mathbb{R}^{2}\backslash B_{R},|u_{n}|\leq K\}}\Big(I_{\mu_{2}}\ast \frac{F(u_{n})}{|x|^{\beta}}\Big)\frac{|u_{n}|^{2}}{|x|^{\beta}}dx\\
\leq& \frac{C(\gamma,K)}{R^{\beta}}\int_{\{x\in\mathbb{R}^{2}\backslash B_{R},|u_{n}|\leq K\}}\Big(I_{\mu_{2}}\ast \frac{F(u_{n})}{|x|^{\beta}}\Big)|u_{n}|^{2}dx.
    \end{aligned}
\end{eqnarray*}
By Lemma $\ref{LEbd}$, we obtain
\begin{eqnarray*}
    \begin{aligned}\displaystyle
    &\int_{\{x\in\mathbb{R}^{2}\backslash B_{R},|u_{n}|\leq K\}}\Big(I_{\mu_{2}}\ast \frac{F(u_{n})}{|x|^{\beta}}\Big)|u_{n}|^{2}dx\\
\leq& \Big(\int_{\mathbb{R}^{2}}\Big(I_{\mu_{2}}\ast \frac{F(u_{n})}{|x|^{\beta}}\Big)\frac{F(u_{n})}{|x|^{\beta}}dx\Big)^{\frac{1}{2}}\Big(\int_{\mathbb{R}^{2}}\Big(I_{\mu_{2}}\ast |u_{n}|^{2}\Big)|u_{n}|^{2}dx\Big)^{\frac{1}{2}}\leq C.
    \end{aligned}
\end{eqnarray*}
Taking $R$ sufficiently large, it follows that
\begin{eqnarray*}
    \begin{aligned}\displaystyle
    \int_{\{x\in\mathbb{R}^{2}\backslash B_{R},|u_{n}|\leq K\}}\Big(I_{\mu_{2}}\ast \frac{F(u_{n})}{|x|^{\beta}}\Big)\frac{F(u_{n})}{|x|^{\beta}}dx\leq \frac{1}{2}\delta.
    \end{aligned}
\end{eqnarray*}

Finally, Since $\delta>0$ is arbitrary, we have
\begin{eqnarray*}
    \begin{aligned}\displaystyle
\int_{\mathbb{R}^{2}}\Big(I_{\mu_{2}}\ast \frac{F(u_{n})}{|x|^{\beta}}\Big)\frac{F(u_{n})}{|x|^{\beta}}dx \rightarrow \int_{\mathbb{R}^{2}}\Big(I_{\mu_{2}}\ast \frac{F(u_{0})}{|x|^{\beta}}\Big)\frac{F(u_{0})}{|x|^{\beta}}dx \ \mbox{as} \ n\rightarrow\infty.
\end{aligned}
\end{eqnarray*}
Repeating the same way, we also get
\begin{eqnarray*}
    \begin{aligned}\displaystyle
\int_{\mathbb{R}^{2}}\Big(I_{\mu_{1}}\ast \frac{G(u_{n})}{|x|^{\alpha}}\Big)\frac{G(u_{n})}{|x|^{\alpha}}dx \rightarrow \int_{\mathbb{R}^{2}}\Big(I_{\mu_{1}}\ast \frac{G(u_{0})}{|x|^{\alpha}}\Big)\frac{G(u_{0})}{|x|^{\alpha}}dx \ \mbox{as} \ n\rightarrow\infty,
    \end{aligned}
\end{eqnarray*}
which completes the proof.
\end{proof}

To find a solution of $(\ref{a})$, we show that the functional $\Phi$ possesses the following geometry.

\begin{lemma}\label{LEcd}Suppose that $(V_{1})-(V_{2}),(H_{0}^{\prime})-(H_{1}^{\prime})$ hold. Then there exists $\rho,\sigma>0$ such that
\begin{eqnarray*}
    \begin{aligned}\displaystyle
    \kappa:=\inf\{\Phi(z):z\in E^{+},\|z\|=\rho\}\geq\sigma>0.
    \end{aligned}
\end{eqnarray*}
\end{lemma}

\begin{proof} Taking $q=3$ in $(\ref{fnonb})$, we know that
\begin{eqnarray*}
    \begin{aligned}\displaystyle
    F(s),G(s)\leq\frac{\epsilon}{2}|s|^{2}+b_{1}|s|^{3}(e^{\gamma s^{2}}-1),\ \ \mbox{for \ all} \ s\in \mathbb{R}.
    \end{aligned}
\end{eqnarray*}
By Proposition $\ref{PRa}$, one has
\begin{eqnarray*}
    \begin{aligned}\displaystyle
    \int\limits_{\mathbb{R}^{2}}(e^{\gamma|u|^{2}}-1)dx&=\int\limits_{\mathbb{R}^{2}}(e^{\gamma\|u\|^{2}(u/\|u\|)^{2}}-1)dx\leq C(M,2\pi), \ \mbox{for\ all}\ \|u\|\leq\sqrt{2\pi/\gamma}.
    \end{aligned}
\end{eqnarray*}
Then we deduce that
\begin{eqnarray*}
    \begin{aligned}\displaystyle
    \int_{\mathbb{R}^{2}}\Big(I_{\mu_{2}}\ast \frac{F(u)}{|x|^{\beta}}\Big)\frac{F(u)}{|x|^{\beta}}dx&\leq|F(u)|_{\frac{4}{2+\mu_{2}-2\beta}}|F(u)|_{\frac{4}{2+\mu_{2}-2\beta}}\leq (\frac{\varepsilon}{2}\|u\|^{2}+C(M,2\pi)|u|_{3}^{3})^{2},
    \end{aligned}
    \end{eqnarray*}
and
\begin{eqnarray*}
    \begin{aligned}\displaystyle
    \int_{\mathbb{R}^{2}}\Big(I_{\mu_{1}}\ast \frac{G(u)}{|x|^{\alpha}}\Big)\frac{G(u)}{|x|^{\alpha}}dx&\leq|G(u)|_{\frac{4}{2+\mu_{1}-2\alpha}}|G(u)|_{\frac{4}{2+\mu_{1}-2\alpha}}
    \leq (\frac{\varepsilon}{2}\|u\|^{2}+C(M,2\pi)|u|_{3}^{3})^{2}.
    \end{aligned}
    \end{eqnarray*}
Together with the continuous embedding $H^{1}_{V}(\mathbb{R}^{2})\hookrightarrow L^{3}(\mathbb{R}^{2})$, taking $\sigma>0$, we obtain
 \begin{eqnarray*}
    \begin{aligned}\displaystyle
      \Phi(z)\geq (\frac{1}{2}-\frac{C\varepsilon}{2})\|u\|^{2}-C\|u\|^{6}\geq\sigma>0.
    \end{aligned}
 \end{eqnarray*}
 If we take $\rho$ to be sufficiently small, the the proof of the lemma is finished.
\end{proof}
Let $e\in H^{1}(\mathbb{R}^{2})\backslash \{0\}$ with $\|e\|=1$ and
\begin{eqnarray*}
    \begin{aligned}\displaystyle
Q=\{r(e,e)+(w,-w):w\in H^{1}(\mathbb{R}^{2}), \|w\|\leq R_{0} \ and \ 0\leq r\leq R_{1}\},
    \end{aligned}
\end{eqnarray*}
where $R_{0},R_{1}>\rho$ are defined in the Lemma $\ref{LEce}$.

\begin{lemma}\label{LEce}
Suppose that $(V_{1})-(V_{2}), (H_{0}^{\prime})-(H_{1}^{\prime})$ hold. Then there exist positive constants $R_{0}$, $R_{1}>\rho$ such that $\sup\Phi(\partial \mathcal{Q})\leq 0$.
\end{lemma}

\begin{proof}
Notice that the boundary $\partial Q$ of the set $Q$ is composed of three parts.

$(i)$ If $\omega\in \partial Q \cap E^{-}$, then $\omega=(u,u)\in E^{-}$, hence
\begin{eqnarray*}
    \begin{aligned}\displaystyle
\Phi(w)= -\|u\|^{2}-\frac{1}{2}\int_{\mathbb{R}^{2}}\Big(I_{\mu_{2}}\ast \frac{F(u)}{|x|^{\beta}}\Big)\frac{F(u)}{|x|^{\beta}}dx-\frac{1}{2}\int_{\mathbb{R}^{2}}\Big(I_{\mu_{1}}\ast \frac{G(-u)}{|x|^{\alpha}}\Big)\frac{G(-u)}{|x|^{\alpha}}dx \leq 0,
    \end{aligned}
\end{eqnarray*}
because $F,G$ are nonnegative functions.

$(ii)$ Suppose that $\omega=r(e,e)+(u,-u)\in\partial Q$ with $\|(u,-u)\|_{E}=R_{0}$ and $0\leq r\leq R_{1}$, we obtain
\begin{eqnarray*}
    \begin{aligned}\displaystyle
\Phi(\omega)=& r^{2}\|e\|^{2}-\|u\|^{2}-\frac{1}{2}\int_{\mathbb{R}^{2}}\Big(I_{\mu_{2}}\ast \frac{F(re+u)}{|x|^{\beta}}\Big)\frac{F(re+u)}{|x|^{\beta}}dx\\
&-\frac{1}{2}\int_{\mathbb{R}^{2}}\Big(I_{\mu_{1}}\ast \frac{G(re-u)}{|x|^{\alpha}}\Big)\frac{G(re-u)}{|x|^{\alpha}}dx
\leq R_{1}^{2}-\frac{1}{2}R_{0}^{2}.
    \end{aligned}
\end{eqnarray*}
Hence $\Phi(\omega)\leq0$ provided that $\sqrt{2}R_{1}\leq R_{0}$.

$(iii)$ If $\omega = R_{1}(e,e)+(u,-u)\in\partial Q$, with $\|(u,-u)\|_{E}\leq R_{0}$ for $R_{0}$ given by case $(ii)$. In this case,
\begin{eqnarray*}
    \begin{aligned}\displaystyle
    \Phi (\omega)=& R_{1}\|e\|^{2}-\|u\|^{2}-\frac{1}{2}\int_{\mathbb{R}^{2}}\Big(I_{\mu_{2}}\ast \frac{F(R_{1}e+u)}{|x|^{\beta}}\Big)\frac{F(R_{1}e+u)}{|x|^{\beta}}dx\\
    &-\frac{1}{2}\int_{\mathbb{R}^{2}}\Big(I_{\mu_{1}}\ast \frac{G(R_{1}e-u)}{|x|^{\alpha}}\Big)\frac{G(R_{1}e-u)}{|x|^{\alpha}}dx.
    \end{aligned}
\end{eqnarray*}

Define
\begin{eqnarray*}
    \begin{aligned}\displaystyle
    A(z)=\int_{\mathbb{R}^{2}}\Big(I_{\mu_{2}}\ast \frac{F(z)}{|x|^{\beta}}\Big)\frac{F(z)}{|x|^{\beta}}dx,
    \end{aligned}
\end{eqnarray*}
and
\begin{eqnarray*}
    \begin{aligned}\displaystyle
    s(t)= H\Big(t\frac{u/R_{1}+e}{\|u\|+\|e\|}\Big).
    \end{aligned}
\end{eqnarray*}
As a same calculation given by Maia in \cite{MMA}, we get
\begin{eqnarray*}
    \begin{aligned}\displaystyle
    \frac{s^{\prime}(t)}{s(t)}\geq\frac{\theta}{t},
    \end{aligned}
\end{eqnarray*}
where $\theta$ is defined in $(H_{2}^{\prime})$, which implies
\begin{eqnarray*}
    \begin{aligned}\displaystyle
    H(u+R_{1}e)=s(R_{1}(\|u\|+\|e\|))\geq C[R_{1}(\|u\|+\|e\|)]^{\theta}.
    \end{aligned}
\end{eqnarray*}
Then, we obtain
\begin{eqnarray*}
    \begin{aligned}\displaystyle
    \Phi (\omega)\leq R_{1}^{2}\|e\|^{2}-\frac{1}{2}H(u+R_{1}e)&\leq R_{1}^{2}\|e\|^{2}-C R_{1}^{\theta}\|y\|^{\theta}\leq R_{1}^{2}-C R_{1}^{\theta}.
    \end{aligned}
\end{eqnarray*}
Since $\theta>2$, taking $R_{1}$ sufficiently large, we obtain $\Phi(\omega)\leq0$.
\end{proof}

\section{{\bfseries Finite dimensional approximation}}\label{FIn}

Since the functional $\Phi$ is strongly indefinite and defined in an infinite dimensional space, standard linking theorems are not applicable. We therefore approximate system $(\ref{a})$ with a sequence of finite dimensional problems (Galerkin approximation procedure).

Associated with the eigenvalues $0<\lambda_{1}<\lambda_{2}<\lambda_{3}<\cdot\cdot\cdot<\lambda_{j} \rightarrow +\infty$ of $(-\Delta +V(x),H^{1}_{V})$, there exists an orthonormal basis $\{\phi_{1},\phi_{2},...\}$ of corresponding eigenfunctions in $H^{1}_{V}$. We set,
\begin{eqnarray*}
    \begin{aligned}\displaystyle
    &E_{n}^{+}= \mbox{span} \{(\phi_{i},\phi_{i})|i=1,...,n\}\\
    &E_{n}^{-}= \mbox{span} \{(\phi_{i},\phi_{-i})|i=1,...,n\}\\
    &E_{n}=E_{n}^{+}\oplus E_{n}^{-}.
    \end{aligned}
\end{eqnarray*}
Let $y\in H^{1}_{V}$ be a fixed nonnegative function and
\begin{eqnarray*}
    \begin{aligned}\displaystyle
Q_{n,y}=\{r(y,y)+w:w\in E_{n}^{-}, \|w\|\leq R_{0} \ and \ 0\leq r\leq R_{1}\},
\end{aligned}
\end{eqnarray*}
where $R_{0},R_{1}$ are given in Lemma $\ref{LEce}$. We recall that these constants depend of $y$ only. We use the following notation:
\begin{eqnarray*}
    \begin{aligned}\displaystyle
H_{n,y}=\mathbb{R}(y,y)\oplus E_{n}, \ \ H_{n,y}^{+}=\mathbb{R}(y,y)\oplus E_{n}^{+}, \ \ H_{n,y}^{-}=\mathbb{R}(y,y)\oplus E_{n}^{-}.
    \end{aligned}
\end{eqnarray*}
Furthermore, define the class of mappings
\begin{eqnarray*}
    \begin{aligned}\displaystyle
\Gamma_{n,y}=\{h\in C(Q_{n,y},H_{n,y}): h(z)=z \ on \ \partial Q_{n,y}\}
    \end{aligned}
\end{eqnarray*}
and set
\begin{eqnarray*}
    \begin{aligned}\displaystyle
    c_{n,y}=\inf\limits_{h\in \Gamma_{n,y}} \max\limits_{z\in Q_{n,y}} \Phi(h(z)).
    \end{aligned}
\end{eqnarray*}

Using an intersection theorem (see \cite[Proposition 5.9]{Rabi}), we obtain
\begin{eqnarray*}
    \begin{aligned}\displaystyle
    h(Q_{n,y})\cap (\partial B_{\rho}\cap E^{+})\neq \emptyset, \ \forall h\in \Gamma_{n,y},
    \end{aligned}
\end{eqnarray*}
which in combination with Lemma $\ref{LEcd}$ implies that $c_{n,y}\geq \sigma>0$. On the other hand, since the identity mapping $Id:Q_{n,y}\rightarrow H_{n,y}$ belongs to $\Gamma_{n,y}$, we have for $z=r(y,y)+(u,-u)\in Q_{n,y}$ that
\begin{eqnarray*}
    \begin{aligned}\displaystyle
    \Phi(z)&=r^{2}\|y\|^{2}-\|u\|^{2}\\
    &-\frac{1}{2}\int_{\mathbb{R}^{2}}\Big(I_{\mu_{2}}\ast \frac{F(ry+u)}{|x|^{\beta}}\Big)\frac{F(ry+u)}{|x|^{\beta}}dx-\frac{1}{2}\int_{\mathbb{R}^{2}}\Big(I_{\mu_{1}}\ast \frac{G(ry-u)}{|x|^{\alpha}}\Big)\frac{G(ry-u)}{|x|^{\alpha}}dx\\
    &\leq r^{2}\|y\|^{2}\leq R_{1}^{2}.
    \end{aligned}
\end{eqnarray*}
Then we have
\begin{eqnarray*}
    \begin{aligned}\displaystyle
    0< \sigma \leq c_{n,y} \leq R_{1}^{2}.
    \end{aligned}
\end{eqnarray*}
We remark that the upper bound does not depend of $n$, but it depends on $y$.

Let us denote by $\Phi_{n,y}$ the functional $\Phi$ restricted to the finite dimensional subspace $H_{n,y}$. Therefore, in view of Lemmas $\ref{LEcd}$ and $\ref{LEce}$, we see that the geometry of a linking theorem holds for the functional $\Phi_{n,y}$. So, applying the linking theorem for $\Phi_{n,y}$ (see Theorem $5.3$ in \cite{Rabi}), we obtain a (PS)$-$sequence, which is bounded in view of Lemma $\ref{LEcb}$. Finally, using the fact that $H_{n,y}$ is a finite dimensional space, we get the main result of this section.
\begin{proposition}\label{PRda}
For each $n\in \mathbb{N}$ and for each $y\in H^{1}$, a fixed nonnegative function, the functional $\Phi_{n,y}$ has a critical point at level $c_{n,y}$. More precisely, there is a $z_{n,y}\in H_{n,y}$ such that
\begin{eqnarray*}
    \begin{aligned}\displaystyle
    \Phi_{n,y}(z_{n,y})=c_{n,y}\in [\sigma, R_{1}^{2}] \ \ and \ \ (\Phi_{n,y})^{\prime}(z_{n,y})=0.
    \end{aligned}
\end{eqnarray*}
Furthermore, $\|z_{n,y}\|\leq C$ where $C$ does not depend of $n$.
\end{proposition}

\section{{\bfseries Proof of Theorem 1.1.}}\label{Subp}
In this section we assume that $g$ has subcritical growth and $f$ has subcritical or critical growth.
{\bf Proof of Theorem} $\ref{THa}$. Applying Proposition $\ref{PRda}$ we have a sequence $z_{n}:=z_{n,y}=(u_{n},v_{n})\in H_{n,y}$ such that
\begin{equation}\label{PRsa}
    \begin{aligned}\displaystyle
    &\|(u_{n},v_{n})\|_{E}\leq C,\\
    &\Phi_{n,y}(u_{n},v_{n})=c_{n}:=c_{n,y}\in [\sigma,R_{1}],\\
    &(\Phi_{n,y})^{\prime}(u_{n},v_{n})=0,\\
    &(u_{n},v_{n})\rightharpoonup (u_{0},v_{0}) \ \ \mbox{in} \ \ E.
    \end{aligned}
\end{equation}
By $(u_{n},v_{n})$ is bounded in $E$, there exists $(u_{0},v_{0})$ such that $(u_{n},v_{n})\rightharpoonup (u_{0},v_{0})$.
Then Taking $(0,\psi)$ and $(\varphi,0)$ as text function in $(\ref{PRsa})$, we get
\begin{eqnarray*}
    \begin{aligned}\displaystyle
    &\int_{\mathbb{R}^{2}}\nabla u_{n}\nabla \psi +V(x)u_{n}\psi dx = \int_{\mathbb{R}^{2}}\Big(I_{\mu_{1}}\ast \frac{G(v_{n})}{|x|^{\alpha}}\Big)\frac{g(v_{n})\psi}{|x|^{\alpha}}dx \ \ \forall \psi \in E_{n}, \\
    &\int_{\mathbb{R}^{2}}\nabla v_{n}\nabla \varphi +V(x)v_{n}\varphi dx = \int_{\mathbb{R}^{2}}\Big(I_{\mu_{2}}\ast \frac{F(u_{n})}{|x|^{\beta}}\Big)\frac{f(u_{n})\varphi}{|x|^{\beta}}dx \ \ \forall \varphi \in E_{n},
    \end{aligned}
\end{eqnarray*}
where $E_{n}:=span\{\phi_{i}|i=1,...,n\}$. By Lemma $\ref{LEcb}$, we have
\begin{eqnarray*}
    \begin{aligned}\displaystyle
    &\lim\limits_{n\rightarrow\infty}\int_{\mathbb{R}^{2}}\Big(I_{\mu_{2}}\ast \frac{F(u_{n})}{|x|^{\beta}}\Big)\frac{f(u_{n})\varphi}{|x|^{\beta}}dx=\int_{\mathbb{R}^{2}}\Big(I_{\mu_{2}}\ast \frac{F(u_{0})}{|x|^{\beta}}\Big)\frac{f(u_{0})\varphi}{|x|^{\beta}}dx,\\
    &\lim\limits_{n\rightarrow\infty}\int_{\mathbb{R}^{2}}\Big(I_{\mu_{1}}\ast \frac{G(v_{n})}{|x|^{\alpha}}\Big)\frac{g(v_{n})\psi}{|x|^{\alpha}}dx=\int_{\mathbb{R}^{2}}\Big(I_{\mu_{1}}\ast \frac{G(v_{0})}{|x|^{\alpha}}\Big)\frac{g(v_{0})\psi}{|x|^{\alpha}}dx.
    \end{aligned}
\end{eqnarray*}
Thus, using the fact that $\cup_{n\in \mathbb{N}}E_{n}$ is dense in $E$, together with Lemma $\ref{LEcc}$, we can obtain that
\begin{equation}\label{PRsb}
    \begin{aligned}\displaystyle
    &\int_{\mathbb{R}^{2}}\nabla u_{0}\nabla \psi +V(x)u_{0}\psi dx = \int_{\mathbb{R}^{2}}\Big(I_{\mu_{1}}\ast \frac{G(v_{0})}{|x|^{\alpha}}\Big)\frac{g(v_{0})\psi}{|x|^{\alpha}}dx \ \ \forall \psi \in E, \\
    &\int_{\mathbb{R}^{2}}\nabla v_{0}\nabla \varphi +V(x)v_{0}\varphi dx = \int_{\mathbb{R}^{2}}\Big(I_{\mu_{2}}\ast \frac{F(u_{0})}{|x|^{\beta}}\Big)\frac{f(u_{0})\varphi}{|x|^{\beta}}dx \ \ \forall \varphi \in E.
    \end{aligned}
\end{equation}
Therefore, we conclude that $\Phi^{\prime}(u_{0},v_{0})=0$ in E, then $(u_{0},v_{0})$ is a critical point of $\Phi$. Then for $\forall \varphi \in C^{\infty}_{0}(\mathbb{R}^{2})$, we know that
\begin{equation}\label{PRsc}
    \begin{aligned}\displaystyle
    \langle u_{0}, \psi\rangle + \langle v_{0}, \varphi\rangle = &\int_{\mathbb{R}^{2}}\Big(I_{\mu_{2}}\ast \frac{F(u_{0})}{|x|^{\beta}}\Big)\frac{f(u_{0})\varphi}{|x|^{\beta}}dx+\int_{\mathbb{R}^{2}}\Big(I_{\mu_{1}}\ast \frac{G(v_{0})}{|x|^{\alpha}}\Big)\frac{g(v_{0})\psi}{|x|^{\alpha}}dx.
    \end{aligned}
\end{equation}

Now, it remains to prove that $u_{0}$, $v_{0}\neq 0$. Assume that $u_{0}\equiv 0$, then by $(\ref{PRsb})$ we know that $v_{0}=0$. Thus the proof of this theorem is divided into two steps.

{\em Step $1.$} $(u_{n},v_{n})\rightarrow (u_{0},v_{0})=(0,0)$, that is, $\|u_{n}\|,\|v_{n}\|\rightarrow 0$.

By the Cauchy-Schwarz inequality, we obtain
\begin{eqnarray*}
    \begin{aligned}\displaystyle
    \lim\limits_{n\rightarrow \infty}\int_{\mathbb{R}^{2}}(\nabla u_{n}\nabla v_{n}+ V(x)u_{n}v_{n})\rightarrow 0,
    \end{aligned}
\end{eqnarray*}
which implies that
\begin{eqnarray*}
    \begin{aligned}\displaystyle
    \int_{\mathbb{R}^{2}}\Big(I_{\mu_{2}}\ast \frac{F(u_{n})}{|x|^{\beta}}\Big)\frac{f(u_{n})u_{n}}{|x|^{\beta}}dx\rightarrow 0 \ and \ \int_{\mathbb{R}^{2}}\Big(I_{\mu_{1}}\ast \frac{G(v_{n})}{|x|^{\alpha}}\Big)\frac{g(v_{n})v_{n}}{|x|^{\alpha}}dx\rightarrow 0.
    \end{aligned}
\end{eqnarray*}
This convergence together with $(H_{3}^{\prime})$ we get that
\begin{eqnarray*}
    \begin{aligned}\displaystyle
    \int_{\mathbb{R}^{2}}\Big(I_{\mu_{2}}\ast \frac{F(u_{n})}{|x|^{\beta}}\Big)\frac{F(u_{n})}{|x|^{\beta}}dx\rightarrow 0 \ and \ \int_{\mathbb{R}^{2}}\Big(I_{\mu_{1}}\ast \frac{G(v_{n})}{|x|^{\alpha}}\Big)\frac{G(v_{n})}{|x|^{\alpha}}dx\rightarrow 0.
    \end{aligned}
\end{eqnarray*}

Then, the last convergence shows us that $c_{n}=0$, a contradiction. Hence, this case cannot occur.

{\em Step $2.$} $(u_{n},v_{n})$ converges weakly to $(u_{0},v_{0})$ in $E$ but does not converge strongly. In other words, $(u_{n},v_{n})\rightharpoonup (0,0)$ in $E$, and there is a constant $a>0$ such that $\liminf\limits_{n\rightarrow\infty}\|u_{n}\|\geq a$ and $\liminf\limits_{n\rightarrow\infty}\|v_{n}\|\geq a$.

Taking $(0,u_{n}),(v_{n},0)$ as test function in $(\ref{PRsa})$, we obtain that
\begin{eqnarray*}
    \begin{aligned}\displaystyle
    \|u_{n}\|^{2}=\int_{\mathbb{R}^{2}}\Big(I_{\mu_{2}}\ast \frac{G(v_{n})}{|x|^{\beta}}\Big)\frac{g(v_{n})u_{n}}{|x|^{\beta}}dx,\quad \mbox{and}\ \
    \|v_{n}\|^{2}=\int_{\mathbb{R}^{2}}\Big(I_{\mu_{1}}\ast \frac{F(u_{n})}{|x|^{\alpha}}\Big)\frac{f(u_{n})v_{n}}{|x|^{\alpha}}dx.
    \end{aligned}
\end{eqnarray*}

Since $g$ is subcritical, we get for all $\gamma>0$,
\begin{equation}\label{PRsd}
    \begin{aligned}\displaystyle
    |g(s)|\leq C_{1}s+C_{2}e^{\gamma s^{2}} \ \ \forall s\in \mathbb{R},
    \end{aligned}
\end{equation}
where $C_{1},C_{2}>0$. Then, using the Hardy-Littlewood-Sobolev and H$\ddot{o}$lder inequalities, we have
\begin{eqnarray*}
    \begin{aligned}\displaystyle
\int_{\mathbb{R}^{2}}|\nabla u_{n}|^{2} +V(x)u^{2} dx = &\int_{\mathbb{R}^{2}}\Big(I_{\mu_{1}}\ast \frac{G(v_{n})}{|x|^{\alpha}}\Big)\frac{g(v_{n})u_{n}}{|x|^{\alpha}}dx\\
\leq& C|G(v_{n})|_{\frac{4}{2+\mu_{1}-2\alpha}} |g(v_{n})u_{n}|_{\frac{4}{2+\mu_{1}-2\alpha}}\\
\leq& C_{1}|v_{n}|_{\frac{8}{2+\mu_{1}-2\alpha}}|u_{n}|_{\frac{8}{2+\mu_{1}-2\alpha}}\\
&+C_{2} |u_{n}|_{\frac{4t^{\prime}}{2+\mu_{1}-2\alpha}}\Big(\int_{\mathbb{R}^{2}}e^{\frac{4\gamma t}{2+\mu_{1}-2\alpha}\|v_{n}\|^{2}(\frac{v_{n}^{2}}{\|v_{n}\|^{2}})}dx\Big)^{\frac{2+\mu_{1}-2\alpha}{4t}},
    \end{aligned}
\end{eqnarray*}
where $t,t^{\prime}>1$ satisfying $\frac{1}{t}+\frac{1}{t^{\prime}}=1$. Together with Trudinger-Moser inequality with $(\ref{PRsd})$, we deduce that
\begin{eqnarray*}
    \begin{aligned}\displaystyle
\Big(\int_{\mathbb{R}^{2}}e^{\frac{4\gamma t}{2+\mu_{1}-2\alpha}\|v_{n}\|^{2}(\frac{v_{n}^{2}}{\|v_{n}\|^{2}})}dx\Big)^{\frac{2+\mu_{1}-2\alpha}{4t}}\leq C_{2},
    \end{aligned}
\end{eqnarray*}
which implies that
\begin{eqnarray*}
    \begin{aligned}\displaystyle
    \int_{\mathbb{R}^{2}}\Big(I_{\mu_{1}}\ast \frac{G(v_{n})}{|x|^{\alpha}}\Big)\frac{g(v_{n})u_{n}}{|x|^{\alpha}}dx \leq C_{1}|v_{n}|_{\frac{8}{2+\mu_{1}-2\alpha}}|u_{n}|_{\frac{8}{2+\mu_{1}-2\alpha}}+ C_{3}|u_{n}|_{\frac{4t^{\prime}}{2+\mu_{1}-2\alpha}}.
    \end{aligned}
\end{eqnarray*}

Using the compact embedding $H^{1}(\mathbb{R}^{2})\hookrightarrow L^{q}(\mathbb{R}^{2})$ for $1\leq q\leq \infty$. Thus, we obtain that
\begin{eqnarray*}
    \begin{aligned}\displaystyle
    \int_{\mathbb{R}^{2}}\Big(I_{\mu_{1}}\ast \frac{G(v_{n})}{|x|^{\alpha}}\Big)\frac{g(v_{n})u_{n}}{|x|^{\alpha}}dx \rightarrow 0 \ \ n\rightarrow \ +\infty.
    \end{aligned}
\end{eqnarray*}
Repeating the same arguments, we obtain that
\begin{eqnarray*}
    \begin{aligned}\displaystyle
    \int_{\mathbb{R}^{2}}\Big(I_{\mu_{2}}\ast \frac{F(u_{n})}{|x|^{\beta}}\Big)\frac{f(u_{n})v_{n}}{|x|^{\beta}}dx \rightarrow 0 \ \ n\rightarrow \ +\infty.
    \end{aligned}
\end{eqnarray*}
Then repeating the same arguments in step $1.$, we get that
\begin{eqnarray*}
    \begin{aligned}\displaystyle
    \int_{\mathbb{R}^{2}}\Big(I_{\mu_{2}}\ast \frac{F(u_{n})}{|x|^{\beta}}\Big)\frac{F(u_{n})}{|x|^{\beta}}dx \rightarrow 0 \ \ and \ \
    \int_{\mathbb{R}^{2}}\Big(I_{\mu_{1}}\ast \frac{G(v_{n})}{|x|^{\alpha}}\Big)\frac{G(v_{n})}{|x|^{\alpha}}dx \rightarrow 0.
    \end{aligned}
\end{eqnarray*}
Thus, we conclude that $c_{n,y}\rightarrow 0$, a contradiction. Consequently, we get a nontrivial weak solution $(u_{0},v_{0})$ of $(\ref{a})$. Then choosing $\psi=u_{0}^{-}=\{-u,0\}$ and $\varphi=0$ in $(\ref{PRsc})$, we obtain
\begin{eqnarray*}
    \begin{aligned}\displaystyle
    -\|u_{0}^{-}\|^{2}=\int_{\mathbb{R}^{2}}\Big(I_{\mu_{1}}\ast\frac{G(v_{0})}{|x|^{\alpha}}\Big)\frac{g(v_{0})u_{0}^{-}}{|x|^{\alpha}}\geq0,
    \end{aligned}
\end{eqnarray*}
which implies that $u_{0}^{-}=0$. Similarly, we deduce $v_{0}^{-}=0$, thus $u_{0},v_{0}\geq0$. If $u_{0}=0$, it is obvious to obtain $v_{0}=0$, allowing us to conclude that $u_{0},v_{0}>0$. Thus, the proof is complete.
\qed\\

\section{{\bfseries The estimates for the critical level}}\label{TECL}
In this section, we assume that $f$ and $g$ have critical growth with exponent critical $\gamma_{0}$, $\alpha=\beta$ and $\mu_{1}=\mu_{2}$. Then together Trudinger-Moser inequality with $(H_{4}^{\prime})$, we obtain an upper bound for the minimax level. In order to do this, we prove the following result in which we combine the Morser type functions and an approximation argument inspired by Tang et al. \cite{QTZA}.

Let us introduce the following Moser type functions supported in $B_{\rho}:=B_{\rho}(0)$ by
\begin{eqnarray*}
    \begin{aligned}\displaystyle
    w_{n}(x)=\frac{1}{\sqrt{2\pi}}
    \left\{ \arraycolsep=1.5pt
       \begin{array}{ll}
        \sqrt{\log n} \ \ \ & 0\leq |x|\leq \rho/n;\\[2mm]
        \frac{\log \frac{\rho}{|x|}}{\sqrt{\log n}} \ \ \ & \rho/n \leq|x|\leq \rho;\\[2mm]
        0,\ \ \ & |x|\geq \rho.
        \end{array}
    \right.
    \end{aligned}
\end{eqnarray*}
One has that
\begin{eqnarray*}
    \|\nabla w_{n}\|^{2}=\int_{\mathbb{R}^{2}}|\nabla w_{n}|^{2}dx =1,
\qquad \mbox{and}\ \ \
    \| w_{n}\|^{2}_{2}=\int_{\mathbb{R}^{2}}| w_{n}|^{2}dx=\delta_{n},
\end{eqnarray*}
where
\begin{equation}\label{ESTa}
    \begin{aligned}\displaystyle
    \delta_{n} =\rho^{2}\Big(\frac{1}{4\log n}-\frac{1}{4n^{2}\log n}-\frac{1}{2n^{2}}\Big)> 0.
    \end{aligned}
\end{equation}

Thanks to $(H_{4}^{\prime})$ we have for all $t \geq R_{\epsilon}$,\\
\begin{equation}\label{ESTb}
    \begin{aligned}\displaystyle
tF(t)\geq (\kappa-\varepsilon) e^{\gamma_{0}t^{2}}\ \ and \ \ tG(t)\geq (\kappa-\varepsilon) e^{\gamma_{0}t^{2}}.
    \end{aligned}
\end{equation}

\begin {lemma}\label{LEfa}Suppose that $(V_{1}),(V_{2}),(H_{0}^{\prime}),(H_{1}^{\prime})$ and $(H_{4}^{\prime})$ hold, then there exists $n_{0}\in  \mathbb{N}$ such that for all $n\geq n_{0}$ the corresponding Moser's function $w_{n}$ satisfies
\begin{equation}\label{ESTc}
    \begin{aligned}\displaystyle
    \sup\limits_{\mathbb{R^{+}}(w_{n},w_{n})\bigoplus \mathbb{E^{-}}} \Phi< \frac{(2-2\alpha+\mu_{1})\pi}{\gamma_{0}}.
    \end{aligned}
\end{equation}
\end{lemma}

\begin{proof}
First, we give the following estimate in $B_{\rho/n}(0)$,
\begin{eqnarray*}
    \begin{aligned}\displaystyle
\int\limits_{B_{\rho/n}(0)}\frac{1}{|y|^{\alpha}}dy\int\limits_{B_{\rho/n}(0)}\frac{dx}{|x|^{\alpha}|x-y|^{2-\mu_{1}}}&\geq (\frac{\rho}{n})^{-2\alpha}\int\limits_{B_{\rho/n}(0)}dx\int\limits_{B_{\rho/n}(0)}\frac{dy}{|x-y|^{2-\mu_{1}}}\\
&\geq (\frac{\rho}{n})^{-2\alpha}\int\limits_{B_{\rho/n}(0)}dx\int\limits_{B_{\rho/n}(x)}\frac{dz}{|z|^{2-\mu_{1}}}\\
&\geq (\frac{\rho}{n})^{-2\alpha} \int\limits_{B_{\rho/n}(0)}dx\int\limits_{B_{\rho/n}(0)-|x|}\frac{dz}{|z|^{2-\mu_{1}}}\\
& =(\frac{\rho}{n})^{-2\alpha}\frac{2\pi}{\mu_{1}}\int\limits_{B_{\rho/n}(0)}(\frac{\rho}{n}-|x|)^{\mu_{1}}\\
& =\frac{4\pi^{2}}{\mu_{1}(1+\mu_{1})(2+\mu_{1})}(\frac{\rho}{n})^{2-2\alpha+\mu_{1}}\\
& =C_{\mu_{1}}(\frac{\rho}{n})^{2-2\alpha+\mu_{1}},
    \end{aligned}
\end{eqnarray*}
where
\begin{eqnarray*}
    \begin{aligned}\displaystyle
    C_{\mu_{1}}=\frac{4\pi^{2}}{\mu_{1}(1+\mu_{1})(2+\mu_{1})}.
    \end{aligned}
\end{eqnarray*}

By $(V_{1}),(V_{2})$ we know that V is bounded locally. Thus we let $V_{\rho}=\max\limits_{{x\in [0,\rho]}}{V(x)}\leq C$. Then by $(H_{4}^{\prime})$, we may choose $\varepsilon>0$ small such that
\begin{equation}\label{import}
    \begin{aligned}\displaystyle
    \log \frac{2A_{\mu_{1}}C_{\mu_{1}}\rho^{Q}(\kappa-\varepsilon)^{2}\gamma_{0}^{2}}{\pi Q (1+\varepsilon)(1+\varepsilon V_{\rho}\rho^{2})}>\frac{V_{\rho}\rho^{2}Q}{4}-1.
    \end{aligned}
\end{equation}
 And using $(\ref{ESTa})$, we have
\begin{eqnarray*}
    \begin{aligned}\displaystyle
\Phi(tw_{n}+v,tw_{n}-v)
    &=t^{2}\|w_{n}\|^{2}-\|v\|^{2}-\Psi(tw_{n}+v,tw_{n}-v)\\
    &\leq t^{2}(\|\nabla w_{n}\|^{2}_{2}+V_{\rho}\|w_{n}\|^{2}_{2})-\Psi(tw_{n}+v,tw_{n}-v)\\
    &\leq (1+V_{\rho}\delta_{n})t^{2}-\Psi(tw_{n}+v,tw_{n}-v),\ \ \ \forall \ t\geq 0.
    \end{aligned}
\end{eqnarray*}
where $\Psi(tw_{n}+v,tw_{n}-v)=\frac{1}{2}\int_{\mathbb{R}^{2}}\Big(I_{\mu_{1}}\ast \frac{F(tw_{n}+v)}{|x|^{\alpha}}\Big)\frac{F(tw_{n}+v)}{|x|^{\beta}}dx+\frac{1}{2}\int_{\mathbb{R}^{2}}\Big(I_{\mu_{1}}\ast \frac{G(tw_{n}-v)}{|x|^{\alpha}}\Big)\frac{G(tw_{n}-v)}{|x|^{\alpha}}dx.$\\
Then let $Q=2-2\alpha+\mu_{1}$, there are four cases to distinguish.

Case $i).$ $t\in[0,\sqrt{\frac{Q\pi}{2\gamma_{0}}}]$. Then by $F,G\geq0$, we have
\begin{eqnarray*}
    \begin{aligned}\displaystyle
    \Phi(tw_{n}+v,tw_{n}-v)
    \leq(1+\frac{V_{\rho}\rho^{2}}{4\log n})t^{2}-\Psi(tw_{n}+v,tw_{n}-v)
    \leq \frac{Q\pi}{2\gamma_{0}}+O(\frac{1}{\log n}),
    \end{aligned}
\end{eqnarray*}
which implies that $(\ref{ESTc})$ hold.

Case $ii).$ $t\in[\sqrt{\frac{Q\pi}{2\gamma_{0}}},\sqrt{\frac{Q\pi}{\gamma_{0}}}]$. In this case, $tw_{n}\geq R_{\varepsilon}$ for $x\in B_{\rho/n}(0)$ and $n\in \mathbb{N}$ large. Then using $(3.1)$, we have\\
\begin{eqnarray*}
    \begin{aligned}\displaystyle
    F(tw_{n}+v)\geq\frac{(\kappa-\varepsilon)e^{\gamma_{0}(tw_{n}+v)^{2}}}{tw_{n}+v} \ \ \ \mbox{and} \ \ \ \ G(tw_{n}-v)\geq\frac{(\kappa-\varepsilon)e^{\gamma_{0}(tw_{n}-v)^{2}}}{tw_{n}-v},
    \end{aligned}
\end{eqnarray*}
which, together with $(\ref{ESTb})$, yields
\begin{eqnarray*}
    \begin{aligned}\displaystyle
    &\Psi(tw_{n}+v,tw_{n}-v)\\
=&\frac{1}{2}\int_{\mathbb{R}^{2}}\Big(I_{\mu_{1}}\ast \frac{F(tw_{n}+v)}{|x|^{\alpha}}\Big)\frac{F(tw_{n}+v)}{|x|^{\alpha}}dx+\frac{1}{2}\int_{\mathbb{R}^{2}}\Big(I_{\mu_{1}}\ast \frac{G(tw_{n}-v)}{|x|^{\alpha}}\Big)\frac{G(tw_{n}-v)}{|x|^{\alpha}}dx\\
\geq& \frac{A_{\mu_{1}}}{2}\int_{B_{\rho/n}}\int_{B_{\rho/n}}\frac{F(tw_{n}+v(x))\cdot F(tw_{n}+v(y))}{|x|^{\alpha}|x-y|^{2-\mu_{1}}|y|^{\alpha}}dxdy\\
    &\quad+\frac{A_{\mu_{1}}}{2}\int_{B_{\rho/n}}\int_{B_{\rho/n}}\frac{G(tw_{n}-v(x))\cdot G(tw_{n}-v(y))}{|x|^{\alpha}|x-y|^{2-\mu_{1}}|y|^{\alpha}}dxdy.
    \end{aligned}
\end{eqnarray*}
Next by the inequality $A(x)e^{x}+B(y)e^{y}\geq 2\sqrt{A(x)B(y)}e^{\frac{x+y}{2}}$, we obtain
\begin{eqnarray*}
    \begin{aligned}\displaystyle
    \Psi(tw_{n}+v,tw_{n}-v)
\geq & \frac{A_{\mu_{1}}(\kappa-\varepsilon)^{2}}{t^{2}w_{n}^{2}}\int_{B_{\rho/n}}\int_{B_{\rho/n}}\frac{e^{2\gamma_{0}t^{2}w_{n}^{2}+v^{2}(x)+v^{2}(y)}}{|x|^{\alpha}|x-y|^{2-\mu_{1}}|y|^{\alpha}}dxdy\\
\geq &\frac{2A_{\mu_{1}}(\kappa-\varepsilon)^{2}\gamma_{0}}{Q\log n}e^{2\gamma_{0}t^{2}w_{n}^{2}}\int_{B_{\rho/n}}\int_{B_{\rho/n}}\frac{1}{|x|^{\alpha}|x-y|^{2-\mu_{1}}|y|^{\alpha}}dxdy\\
\geq & \frac{2A_{\mu_{1}}(\kappa-\varepsilon)\gamma_{0}C_{\mu_{1}}}{Q\log n}\Big(\frac{\rho}{n}\Big)^{2-2\alpha+\mu_{1}}e^{2\gamma_{0}t^{2}w_{n}^{2}}\\
\geq &\frac{2A_{\mu_{1}}C_{\mu_{1}}\rho^{Q}(\kappa-\varepsilon)^{2}\gamma_{0}}{Q n^{Q}\log n}e^{(\pi)^{-1}\gamma_{0}t^{2}\log n}.
    \end{aligned}
\end{eqnarray*}
Then we obtain that
\begin{eqnarray*}
    \begin{aligned}\displaystyle
    \Phi(tw_{n}+v,tw_{n}-v)
     \leq (1+\frac{V_{\rho}\rho^{2}}{4\log n})t^{2} -\frac{2A_{\mu_{1}}C_{\mu_{1}}\rho^{Q}(\kappa-\varepsilon)^{2}\gamma_{0}}{Q n^{Q}\log n}e^{(\pi)^{-1}\gamma_{0}t^{2}\log n}:=\varphi_{n}(t).
    \end{aligned}
\end{eqnarray*}
Choosing $t_{n}>0$ satisfying $\varphi_{n}^{\prime}(t_{n})=0$, then we obtain
\begin{eqnarray*}
    \begin{aligned}\displaystyle
    1+\frac{V_{\rho}\rho^{2}}{4\log n}=\frac{2A_{\mu_{1}}C_{\mu_{1}}\rho^{Q}(\kappa-\varepsilon)^{2}\gamma_{0}^{2}}{ Q \pi n^{Q}}e^{(\pi)^{-1}\gamma_{0}t^{2}\log n}.
    \end{aligned}
\end{eqnarray*}
It follows that
\begin{equation}\label{ESTd}
    \begin{aligned}\displaystyle
    t_{n}^{2}&=\frac{Q\pi}{\gamma_{0}}\Big[1+\frac{\log(1+\frac{V_{\rho}\rho^{2}}{4\log n})-\log\frac{2A_{\mu_{1}}C_{\mu_{1}}\rho^{Q}(\kappa-\varepsilon)^{2}\gamma_{0}^{2}}{\pi Q}}{Q\log n}\Big]\\
    &\leq \frac{Q\pi}{\gamma_{0}}-\frac{\pi}{\gamma_{0}\log n}\log \frac{2A_{\mu_{1}}C_{\mu_{1}}\rho^{Q}(\kappa-\varepsilon)^{2}\gamma_{0}^{2}}{\pi Q  (1+\varepsilon V_{\rho}\rho^{2})},
    \end{aligned}
\end{equation}
and
\begin{eqnarray*}
    \begin{aligned}\displaystyle
    \varphi_{n}(t)\leq \varphi_{n}(t_{n})=(1+\frac{V_{\rho}\rho^{2}}{4\log n})t_{n}^{2} -\frac{\pi }{\gamma_{0}\log n}(1+\frac{V_{\rho}\rho^{2}}{4\log n}), \ \ \forall \ t\geq 0.
    \end{aligned}
\end{eqnarray*}
Using $(\ref{ESTd})$, we get
\begin{eqnarray*}
    \begin{aligned}\displaystyle
    \varphi_{n}(t)&\leq (1+\frac{V_{\rho}\rho^{2}}{4\log n})t_{n}^{2}-\frac{\pi}{\gamma_{0}\log n}(1+\frac{V_{\rho}\rho^{2}}{4\log n})\\
    &\leq (1+\frac{V_{\rho}\rho^{2}}{4\log n})\Big[\frac{Q\pi}{\gamma_{0}}-\frac{\pi}{\gamma_{0}\log n}\log \frac{2A_{\mu_{1}}C_{\mu_{1}}\rho^{Q}(\kappa-\varepsilon)^{2}\gamma_{0}}{\pi Q  (1+\varepsilon V_{\rho}\rho^{2})}\Big]-\frac{\pi}{\gamma_{0}\log n}+O(\frac{1}{\log^{2} n})\\
    &\leq \frac{Q\pi}{\gamma_{0}}+\frac{V_{\rho}\rho^{2}\pi Q}{4\gamma_{0}\log n}-\frac{\pi}{\gamma_{0}\log n}-\frac{\pi}{\gamma_{0}\log n}\log \frac{2A_{\mu_{1}}C_{\mu_{1}}\rho^{Q}(\kappa-\varepsilon)^{2}\gamma_{0}^{2}}{\pi Q  (1+\varepsilon V_{\rho}\rho^{2})}+O(\frac{1}{\log^{2} n}).
    \end{aligned}
\end{eqnarray*}
Hence, we obtain
\begin{eqnarray*}
    \begin{aligned}\displaystyle
    \Phi(tw_{n}+v,tw_{n}-v)\leq \frac{Q\pi}{\gamma_{0}}+\frac{\pi}{\gamma_{0}\log n}\Big[\frac{V_{\rho}\rho^{2}Q}{4}-1-\log \frac{2A_{\mu_{1}}C_{\mu_{1}}\rho^{Q}(\kappa-\varepsilon)^{2}\gamma_{0}^{2}}{\pi Q  (1+\varepsilon V_{\rho}\rho^{2})}\Big]+O(\frac{1}{\log^{2} n}),
    \end{aligned}
\end{eqnarray*}
which together with $(\ref{import})$ imply that $(\ref{ESTc})$ hold.

Case $iii).$ $t\in\Big[\sqrt{\frac{Q\pi}{\gamma_{0}}},\sqrt{\frac{Q\pi}{\gamma_{0}}(1+\varepsilon)}\Big]$. In this case, $tw_{n}\geq R_{\varepsilon}$ for $x\in B_{\rho/n}(0)$ and $n\in \mathbb{N}$ large. Then using $(\ref{ESTb})$, we have
\begin{eqnarray*}
    \begin{aligned}\displaystyle
    \Psi(tw_{n}+v,tw_{n}-v)
    &\geq \frac{A_{\mu_{1}}(\kappa-\varepsilon)^{2}}{t^{2}w_{n}^{2}}\int_{B_{\rho/n}}\int_{B_{\rho/n}}\frac{e^{2\gamma_{0}t^{2}w_{n}^{2}+v^{2}(x)+v^{2}(y)}}{|x|^{\alpha}|x-y|^{2-\mu_{1}}|y|^{\alpha}}dxdy\\
    &\geq \frac{2A_{\mu_{1}}C_{\mu_{1}}\rho^{Q}(\kappa-\varepsilon)^{2}\gamma_{0}}{(1+\varepsilon)Q n^{Q}\log n}e^{(\pi)^{-1}\gamma_{0}t^{2}\log n}.
    \end{aligned}
\end{eqnarray*}
Then
\begin{eqnarray*}
    \begin{aligned}\displaystyle
    \Phi(tw_{n}+v,tw_{n}-v)
    &\leq (1+\frac{V_{\rho}\rho^{2}}{4\log n})t^{2} -\frac{2A_{\mu_{1}}C_{\mu_{1}}\rho^{Q}(\kappa-\varepsilon)^{2}\gamma_{0}}{(1+\varepsilon)Q n^{Q}\log n}e^{(\pi)^{-1}\gamma_{0}t^{2}\log n}:=\psi_{n}(t).
    \end{aligned}
\end{eqnarray*}
Choosing $\widetilde{t}_{n}>0$ satisfying $\psi_{n}^{\prime}(\widetilde{t}_{n})=0$, then we obtain
\begin{eqnarray*}
    \begin{aligned}\displaystyle
    1+\frac{V_{\rho}\rho^{2}}{4\log n}=\frac{2A_{\mu_{1}}C_{\mu_{1}}\rho^{Q}(\kappa-\varepsilon)^{2}\gamma_{0}^{2}}{(1+\varepsilon)\pi Q n^{Q}}e^{(\pi)^{-1}\gamma_{0}t^{2}\log n},
    \end{aligned}
\end{eqnarray*}
It follows that
\begin{equation}\label{ESTe}
    \begin{aligned}\displaystyle
    \widetilde{t}_{n}^{2}&=\frac{Q\pi}{\gamma_{0}}\Big[1+\frac{\log(1+\frac{V_{\rho}\rho^{2}}{4\log n})-\frac{2A_{\mu_{1}}C_{\mu_{1}}\rho^{Q}(\kappa-\varepsilon)^{2}\gamma_{0}^{2}}{(1+\varepsilon)\pi Q }}{Q\log n}\Big]\\
    &\leq \frac{Q\pi}{\gamma_{0}}+\frac{\pi}{\gamma_{0}\log n}\log \frac{\pi Q (1+\varepsilon)(1+\varepsilon V_{\rho}\rho^{2})}{2A_{\mu_{1}}C_{\mu_{1}}\rho^{Q}(\kappa-\varepsilon)^{2}\gamma_{0}^{2}},
    \end{aligned}
\end{equation}
and
\begin{eqnarray*}
    \begin{aligned}\displaystyle
    \psi_{n}(t)\leq \psi_{n}(\widetilde{t}_{n})=(1+\frac{V_{\rho}\rho^{2}}{4\log n})\widetilde{t}_{n}^{2} -\frac{\pi }{\gamma_{0}\log n}(1+\frac{V_{\rho}\rho^{2}}{4\log n}), \ \ \forall \ t \ \geq \ 0.
    \end{aligned}
\end{eqnarray*}
Using $(\ref{ESTe})$, we get
\begin{eqnarray*}
    \begin{aligned}\displaystyle
    \psi_{n}(t)&\leq (1+\frac{V_{\rho}\rho^{2}}{4\log n})\widetilde{t}_{n}^{2}-\frac{\pi }{\gamma_{0}\log n}(1+\frac{V_{\rho}\rho^{2}}{4\log n})\\
    &\leq (1+\frac{V_{\rho}\rho^{2}}{4\log n})\Big[\frac{Q\pi}{\gamma_{0}}+\frac{\pi}{\gamma_{0}\log n}\log \frac{\pi Q (1+\varepsilon)(1+\varepsilon V_{\rho}\rho^{2})}{2A_{\mu_{1}}C_{\mu_{1}}\rho^{Q}(\kappa-\varepsilon)^{2}\gamma_{0}^{2}}\Big]-\frac{\pi }{\gamma_{0}\log n}+O(\frac{1}{\log^{2} n})\\
    &\leq \frac{Q\pi}{\gamma_{0}}+\frac{\pi}{\gamma_{0}\log n}\Big[ \frac{V_{\rho}\rho^{2}Q}{4}-1-\log \frac{2A_{\mu_{1}}C_{\mu_{1}}\rho^{Q}(\kappa-\varepsilon)^{2}\gamma_{0}^{2}}{\pi Q (1+\varepsilon)(1+\varepsilon V_{\rho}\rho^{2})}\Big]+O(\frac{1}{\log^{2} n}).
    \end{aligned}
\end{eqnarray*}
Hence, we obtain
\begin{eqnarray*}
    \begin{aligned}\displaystyle
    \Phi(tw_{n}+v,tw_{n}-v)\leq \frac{Q\pi}{\gamma_{0}}+\frac{\pi}{\gamma_{0}\log n}\Big[ \frac{V_{\rho}\rho^{2}Q}{4}-1-\log \frac{2A_{\mu_{1}}C_{\mu_{1}}\rho^{Q}(\kappa-\varepsilon)^{2}\gamma_{0}^{2}}{\pi Q (1+\varepsilon)(1+\varepsilon V_{\rho}\rho^{2})}\Big]+O(\frac{1}{\log^{2} n}),
    \end{aligned}
\end{eqnarray*}
which together with $(\ref{import})$ imply that $(\ref{ESTc})$ hold.

Case $iv).$ $t\in\Big(\sqrt{\frac{Q\pi}{\gamma_{0}}(1+\varepsilon)},+\infty\Big)$. In this case, $tw_{n}\geq R_{\varepsilon}$ for $x\in B_{\rho/n}(0)$ and $n\in \mathbb{N}$ large. Then using $(\ref{ESTb})$, we have
\begin{eqnarray*}
    \begin{aligned}\displaystyle
    \Phi(tw_{n}+v,tw_{n}-v)
    &\leq(1+\frac{V_{\rho}\rho^{2}}{4\log n})t^{2}-\Psi(tw_{n}+v,tw_{n}-v)\\
    &\leq(1+\frac{V_{\rho}\rho^{2}}{4\log n})t^{2}-\frac{2\pi A_{\mu_{1}} C_{\mu_{1}} \rho^{Q} (\kappa-\varepsilon)^{2}}{  t^{2} n^{Q}\log n}e^{(\pi)^{-1}\gamma_{0}t^{2}\log n}\\
    &\leq (1+\frac{V_{\rho}\rho^{2}}{4\log n})\frac{Q\pi(1+\varepsilon)}{\gamma_{0}}-\frac{2 A_{\mu_{1}} C_{\mu_{1}} \rho^{Q} (\kappa-\varepsilon)^{2}}{  (1+\varepsilon)^{2} Q\log n}e^{Q\varepsilon \log n}\\
    &< \frac{Q\pi}{\gamma_{0}},
    \end{aligned}
\end{eqnarray*}
where we have used the fact that the function
\begin{eqnarray*}
    \begin{aligned}\displaystyle
\phi(t)=(1+\frac{V_{\rho}\rho^{2}}{4\log n})t^{2}-\frac{2\pi A_{\mu_{1}} C_{\mu_{1}} \rho^{Q} (\kappa-\varepsilon)^{2}}{  t^{2} n^{Q}\log n}e^{(\pi)^{-1}\gamma_{0}t^{2}\log n}
    \end{aligned}
\end{eqnarray*}
is decreasing on $t\in\Big(\sqrt{\frac{Q\pi}{\gamma_{0}}(1+\varepsilon)},+\infty\Big)$. Thus $(\ref{ESTc})$ holds for $n\geq n_{0}$.\\
Then the proof is complete.
\end{proof}

\section{{\bfseries Proof of Theorem 1.2.}}\label{PROOF}

In this section we assume that $f$ and $g$ have critical growth with exponent critical $\gamma_{0}$, $\alpha=\beta$ and $\mu_{1}=\mu_{2}$.\\
{\bf Proof of Theorem} $\ref{THb}$. The Lemma $\ref{LEfa}$ implies that there is $\delta>0$ such that
\begin{eqnarray*}
    \begin{aligned}\displaystyle
    c_{n}:=c_{n,y}\leq \frac{Q\pi}{\gamma_{0}}-\delta.
\end{aligned}
\end{eqnarray*}
Applying Proposition $\ref{PRda}$ we have a sequence $z_{n}:=z_{n,y}=(u_{n},v_{n})\in H_{n,y}$ such that
\begin{equation}\label{PRCa}
    \begin{aligned}\displaystyle
    &\|(u_{n},v_{n})\|_{E}\leq C,\\
    &\Phi_{n,y}(u_{n},v_{n})=c_{n}:=c_{n,y}\in [\sigma,\frac{Q\pi}{\gamma_{0}}-\delta),\\
    &(\Phi_{n,y})^{\prime}(u_{n},v_{n})=0,\\
    &(u_{n},v_{n})\rightharpoonup (u_{0},v_{0}) \ \ \mbox{in} \ \ E.
    \end{aligned}
\end{equation}

By $(u_{n},v_{n})$ is bounded in $E$, $\mu_{1}=\mu_{2}$ and $\alpha=\beta$, there exists $(u_{0},v_{0})$ such that $(u_{n},v_{n})\rightharpoonup (u_{0},v_{0})$.
Then Taking $(0,\psi)$ and $(\varphi,0)$ as text function in $(\ref{PRCa})$, we get
\begin{eqnarray*}
    \begin{aligned}\displaystyle
    &\int_{\mathbb{R}^{2}}\nabla u_{n}\nabla \psi +V(x)u_{n}\psi dx = \int_{\mathbb{R}^{2}}\Big(I_{\mu_{1}}\ast \frac{G(v_{n})}{|x|^{\alpha}}\Big)\frac{g(v_{n})\psi}{|x|^{\alpha}}dx \ \ \forall \psi \in E_{n}, \\
    &\int_{\mathbb{R}^{2}}\nabla v_{n}\nabla \varphi +V(x)v_{n}\varphi dx = \int_{\mathbb{R}^{2}}\Big(I_{\mu_{1}}\ast \frac{F(u_{n})}{|x|^{\alpha}}\Big)\frac{f(u_{n})\varphi}{|x|^{\alpha}}dx \ \ \forall \varphi \in E_{n},
    \end{aligned}
\end{eqnarray*}
where $E_{n}:=span\{\phi_{i}|i=1,...,n\}$. By Lemma $\ref{LEcb}$, we have
\begin{eqnarray*}
    \begin{aligned}\displaystyle
    &\lim\limits_{n\rightarrow\infty}\int_{\mathbb{R}^{2}}\Big(I_{\mu_{1}}\ast \frac{F(u_{n})}{|x|^{\alpha}}\Big)\frac{f(u_{n})\varphi}{|x|^{\alpha}}dx=\int_{\mathbb{R}^{2}}\Big(I_{\mu_{1}}\ast \frac{F(u_{0})}{|x|^{\alpha}}\Big)\frac{f(u_{0})\varphi}{|x|^{\alpha}}dx,\\
    &\lim\limits_{n\rightarrow\infty}\int_{\mathbb{R}^{2}}\Big(I_{\mu_{1}}\ast \frac{G(v_{n})}{|x|^{\alpha}}\Big)\frac{g(v_{n})\psi}{|x|^{\alpha}}dx=\int_{\mathbb{R}^{2}}\Big(I_{\mu_{1}}\ast \frac{G(v_{0})}{|x|^{\alpha}}\Big)\frac{g(v_{0})\psi}{|x|^{\alpha}}dx.
    \end{aligned}
\end{eqnarray*}
Thus, using the fact that $\cup_{n\in \mathbb{N}}E_{n}$ is dense in $E$, together with Lemma $\ref{LEcc}$, we can obtain that
\begin{equation}\label{PRCb}
    \begin{aligned}\displaystyle
    &\int_{\mathbb{R}^{2}}\nabla u_{0}\nabla \psi +V(x)u_{0}\psi dx = \int_{\mathbb{R}^{2}}\Big(I_{\mu_{1}}\ast \frac{G(v_{0})}{|x|^{\alpha}}\Big)\frac{g(v_{0})\psi}{|x|^{\alpha}}dx \ \ \forall \psi \in E, \\
    &\int_{\mathbb{R}^{2}}\nabla v_{0}\nabla \varphi +V(x)v_{0}\varphi dx = \int_{\mathbb{R}^{2}}\Big(I_{\mu_{1}}\ast \frac{F(u_{0})}{|x|^{\alpha}}\Big)\frac{f(u_{0})\varphi}{|x|^{\alpha}}dx \ \ \forall \varphi \in E.
    \end{aligned}
\end{equation}
Therefore, we conclude that $\Phi^{\prime}(u_{0},v_{0})=0$ in E, then $(u_{0},v_{0})$ is a critical point of $\Phi$. Then for $\forall \psi, \varphi \in C^{\infty}_{0}(\mathbb{R}^{2})$, we know that
\begin{equation}\label{PRCc}
    \begin{aligned}\displaystyle
    \langle u_{0}, \psi\rangle + \langle v_{0}, \varphi\rangle = &\int_{\mathbb{R}^{2}}\Big(I_{\mu_{1}}\ast \frac{F(u_{0})}{|x|^{\alpha}}\Big)\frac{f(u_{0})\varphi}{|x|^{\alpha}}dx+\int_{\mathbb{R}^{2}}\Big(I_{\mu_{1}}\ast \frac{G(v_{0})}{|x|^{\alpha}}\Big)\frac{g(v_{0})\psi}{|x|^{\alpha}}dx.
    \end{aligned}
\end{equation}

Now, it remains to prove that $u_{0}$, $v_{0}\neq 0$. Assume that $u_{0}\equiv 0$, then by $(\ref{PRCc})$ we know that $v_{0}=0$. Thus the proof of this theorem is divided into two steps.

Step $1.$ $(u_{n},v_{n})\rightarrow (u_{0},v_{0})=(0,0)$, that is, $\|u_{n}\|,\|v_{n}\|\rightarrow 0$. By the Cauchy-Schwarz inequality, we obtain
\begin{eqnarray*}
    \begin{aligned}\displaystyle
    \lim\limits_{n\rightarrow \infty}\int_{\mathbb{R}^{2}}(\nabla u_{n}\nabla v_{n}+ V(x)u_{n}v_{n})\rightarrow 0.
    \end{aligned}
\end{eqnarray*}
Which implies that
\begin{eqnarray*}
    \begin{aligned}\displaystyle
    \int_{\mathbb{R}^{2}}\Big(I_{\mu_{1}}\ast \frac{F(u_{n})}{|x|^{\alpha}}\Big)\frac{f(u_{n})u_{n}}{|x|^{\alpha}}dx\rightarrow 0 \ and \ \int_{\mathbb{R}^{2}}\Big(I_{\mu_{1}}\ast \frac{G(v_{n})}{|x|^{\alpha}}\Big)\frac{g(v_{n})v_{n}}{|x|^{\alpha}}dx\rightarrow 0,
    \end{aligned}
\end{eqnarray*}
together with $(H_{3})$ we get that
\begin{eqnarray*}
    \begin{aligned}\displaystyle
    \int_{\mathbb{R}^{2}}\Big(I_{\mu_{1}}\ast \frac{F(u_{n})}{|x|^{\alpha}}\Big)\frac{F(u_{n})}{|x|^{\alpha}}dx\rightarrow 0 \ and \ \int_{\mathbb{R}^{2}}\Big(I_{\mu_{1}}\ast \frac{G(v_{n})}{|x|^{\alpha}}\Big)\frac{G(v_{n})}{|x|^{\alpha}}dx\rightarrow 0.
    \end{aligned}
\end{eqnarray*}

Then, the last convergence shows us that $c_{n}=0$, a contradiction. Hence, this case cannot occur.

Step $2.$ $(u_{n},v_{n})$ converges weakly to $(u_{0},v_{0})$ in $E$ but does not converge strongly. In other words, $(u_{n},v_{n})\rightharpoonup (0,0)$ in $E$, and there is a constant $a>0$ such that $\liminf\limits_{n\rightarrow\infty}\|u_{n}\|\geq a$ and $\liminf\limits_{n\rightarrow\infty}\|v_{n}\|\geq a$.

Taking $(0,u_{n}),(v_{n},0)$ as test function in $(\ref{PRCa})$, we obtain that
\begin{eqnarray*}
    \begin{aligned}\displaystyle
    \|u_{n}\|^{2}=\int_{\mathbb{R}^{2}}\Big(I_{\mu_{1}}\ast \frac{G(v_{n})}{|x|^{\alpha}}\Big)\frac{g(v_{n})u_{n}}{|x|^{\alpha}}dx,\quad
    \mbox{and}\ \ \|v_{n}\|^{2}=\int_{\mathbb{R}^{2}}\Big(I_{\mu_{1}}\ast \frac{F(u_{n})}{|x|^{\alpha}}\Big)\frac{f(u_{n})v_{n}}{|x|^{\alpha}}dx.
    \end{aligned}
\end{eqnarray*}

For $\delta > 0$ sufficiently small, we set
\begin{eqnarray*}
    \begin{aligned}\displaystyle
    \overline{u_{n}}=(\frac{Q\pi}{\gamma_{0}}-\delta)^{\frac{1}{2}}\frac{u_{n}}{\|u_{n}\|},\quad \ \ \overline{v_{n}}=(\frac{Q\pi}{\gamma_{0}}-\delta)^{\frac{1}{2}}\frac{v_{n}}{\|v_{n}\|},
    \end{aligned}
\end{eqnarray*}
thus we have
\begin{eqnarray*}
    \begin{aligned}\displaystyle
    (\frac{Q\pi}{\gamma_{0}}-\delta)^{\frac{1}{2}}\|u_{n}\|=\int_{\mathbb{R}^{2}}\Big(I_{\mu}\ast \frac{G(v_{n})}{|x|^{\alpha}}\Big)\frac{g(v_{n})\overline{u_{n}}}{|x|^{\alpha}}dx,
    \end{aligned}
\end{eqnarray*}
and
\begin{eqnarray*}
    \begin{aligned}\displaystyle
(\frac{Q\pi}{\gamma_{0}}-\delta)^{\frac{1}{2}}\|v_{n}\|=\int_{\mathbb{R}^{2}}\Big(I_{\mu_{1}}\ast \frac{F(u_{n})}{|x|^{\alpha}}\Big)\frac{f(u_{n})\overline{v_{n}}}{|x|^{\alpha}}dx.
    \end{aligned}
\end{eqnarray*}

Following the same steps used in \cite{MMA}, we are able to obtain that
\begin{eqnarray*}
    \begin{aligned}\displaystyle
    (\frac{Q\pi}{\gamma_{0}}-\delta)^{\frac{1}{2}}\|u_{n}\|\leq (1+\frac{\varepsilon}{\gamma_{0}})^{\frac{1}{2}}\int_{\mathbb{R}^{2}}\Big(I_{\mu_{1}}\ast \frac{G(v_{n})}{|x|^{\alpha}}\Big)\frac{g(v_{n})v_{n}}{|x|^{\alpha}}dx+o(1),\\
    (\frac{Q\pi}{\gamma_{0}}-\delta)^{\frac{1}{2}}\|v_{n}\|\leq (1+\frac{\varepsilon}{\gamma_{0}})^{\frac{1}{2}}\int_{\mathbb{R}^{2}}\Big(I_{\mu_{1}}\ast \frac{F(u_{n})}{|x|^{\alpha}}\Big)\frac{f(u_{n})u_{n}}{|x|^{\alpha}}dx+o(1).
    \end{aligned}
\end{eqnarray*}

Next we claim that
\begin{eqnarray*}
    \begin{aligned}\displaystyle
    \int_{\mathbb{R}^{2}}\Big(I_{\mu_{1}}\ast \frac{F(u_{n})}{|x|^{\alpha}}\Big)\frac{F(u_{n})}{|x|^{\alpha}}dx\rightarrow 0 \ \  \mbox{and}\ \
    \int_{\mathbb{R}^{2}}\Big(I_{\mu_{1}}\ast \frac{G(v_{n})}{|x|^{\alpha}}\Big)\frac{g(v_{n})v_{n}}{|x|^{\alpha}}dx\rightarrow 0,
    \end{aligned}
\end{eqnarray*}
when $(u_{n},v_{n})\rightharpoonup (0,0)$.

From $(\ref{PRCa})$, $c_{n}\leq \frac{Q\pi}{\gamma_{0}}-\delta$ implies that
\begin{eqnarray*}
    \begin{aligned}\displaystyle
    |\int_{\mathbb{R}^{2}}(\nabla u_{n}\nabla v_{n}+V(x)u_{n}v_{n})|\leq \frac{Q\pi}{\gamma_{0}}-\delta +o(1).
    \end{aligned}
\end{eqnarray*}
Then
\begin{eqnarray*}
    \begin{aligned}\displaystyle
    \|u_{n}\|+\|v_{n}\| &\leq (1+\frac{\varepsilon}{\gamma_{0}})^{\frac{1}{2}}\int_{\mathbb{R}^{2}}\Big(I_{\mu_{1}}\ast \frac{G(v_{n})}{|x|^{\alpha}}\Big)\frac{g(v_{n})v_{n}}{|x|^{\alpha}}dx\\
    &+(1+\frac{\varepsilon}{\gamma_{0}})^{\frac{1}{2}}\int_{\mathbb{R}^{2}}\Big(I_{\mu_{1}}\ast \frac{F(u_{n})}{|x|^{\alpha}}\Big)\frac{f(u_{n})u_{n}}{|x|^{\alpha}}dx\\
    &\leq (1+\frac{\varepsilon}{\gamma_{0}})^{\frac{1}{2}}(\frac{Q\pi}{\gamma_{0}}-\delta)^{\frac{1}{2}}\\
    &\leq 2(\frac{Q\pi}{\gamma_{0}}-\delta)^{\frac{1}{2}},
    \end{aligned}
\end{eqnarray*}
for $\varepsilon>0$ sufficiently small and $n$ sufficiently large.\\
Then, without loss of generality, suppose that
\begin{eqnarray*}
    \begin{aligned}\displaystyle
    \|u_{n}\|\leq (\frac{Q\pi}{\gamma_{0}}-\delta)^{\frac{1}{2}}.
    \end{aligned}
\end{eqnarray*}

Using the compact embedding $H^{1}(\mathbb{R}^{2})\hookrightarrow L^{q}(\mathbb{R}^{2})$ for $1\leq q<\infty$, we obtain $u_{n},v_{n}\rightarrow 0$ in $L^{q}(\mathbb{R}^{2})$ for $1\leq q<\infty$.

Taking $t>1$ with $(H_{4}^{\prime})$ such that
\begin{equation}\label{PRCd}
    \begin{aligned}\displaystyle
    (\gamma_{0}+\varepsilon)(\frac{Q\pi}{\gamma_{0}}-\delta)(\frac{4}{2+\mu_{1}-2\alpha})t<4\pi, \end{aligned}
\end{equation}
Then, using the Hardy-Littlewood-Sobolev and $H\ddot{o}lder$ inequalities, we get
\begin{eqnarray*}
    \begin{aligned}\displaystyle
    \int_{\mathbb{R}^{2}}\Big(I_{\mu_{1}}\ast \frac{F(u_{n})}{|x|^{\alpha}}\Big)\frac{f(u_{n})u_{n}}{|x|^{\alpha}}dx\leq C|F(u_{n})|_{\frac{4}{2+\mu_{1}-2\alpha}}|f(u_{n})u_{n}|_{\frac{4}{2+\mu_{1}-2\alpha}},
    \end{aligned}
\end{eqnarray*}
where $C$ is a positive constant. Then together with $(\ref{fnon})$, we obtain
\begin{eqnarray*}
    \begin{aligned}\displaystyle
    |f(s)|\leq\varepsilon |s|+b_{1}|s|^{q-1}(e^{(\gamma_{0}+\varepsilon) s^{2}}-1) \ \ \ \forall \ s \in \mathbb{R}.
    \end{aligned}
\end{eqnarray*}
Then,
\begin{eqnarray*}
    \begin{aligned}\displaystyle
    |f(u_{n})u_{n}|_{\frac{4}{2+\mu_{1}-2\alpha}}\leq\varepsilon |u_{n}|^{2}_{2}+C |u_{n}|^{\frac{2+\mu_{1}-2\alpha}{4t^{\prime}}}_{\frac{4qt^{\prime}}{2+\mu_{1}-2\alpha}}\Big(\int_{\mathbb{R}^{2}}[e^{\frac{4(\gamma_{0}+\varepsilon) t}{2+\mu_{1}-2\alpha}\|u\|^{2}(\frac{u_{n}^{2}}{\|u\|^{2}})}-1]dx\Big)^{\frac{2+\mu_{1}-2\alpha}{4t}},
    \end{aligned}
\end{eqnarray*}
where $t,t^{\prime}>1$ satisfying $\frac{1}{t}+\frac{1}{t^{\prime}}=1$. Together Trudinger-Moser inequality with $(\ref{PRCd})$, we deduce that
\begin{eqnarray*}
    \begin{aligned}\displaystyle
    \Big(\int_{\mathbb{R}^{2}}[e^{\frac{4(\gamma_{0}+\varepsilon) t}{2+\mu_{1}-2\alpha}\|u\|^{2}(\frac{u_{n}^{2}}{\|u\|^{2}})}-1]dx\Big)^{\frac{2+\mu_{1}-2\alpha}{4t}}\leq C_{1},
    \end{aligned}
\end{eqnarray*}
which implies that
\begin{eqnarray*}
    \begin{aligned}\displaystyle
    \int_{\mathbb{R}^{2}}\Big(I_{\mu_{1}}\ast \frac{F(u_{n})}{|x|^{\alpha}}\Big)\frac{f(u_{n})u_{n}}{|x|^{\alpha}}dx\leq \varepsilon^{2} |u_{n}|^{4}_{2}+ C_{2} |u_{n}|^{\frac{2+\mu_{1}-2\alpha}{2t^{\prime}}}_{\frac{4qt^{\prime}}{2+\mu_{1}-2\alpha}}.
    \end{aligned}
\end{eqnarray*}
By compact embedding, we get
\begin{eqnarray*}
    \begin{aligned}\displaystyle
    \int_{\mathbb{R}^{2}}\Big(I_{\mu_{1}}\ast \frac{F(u_{n})}{|x|^{\alpha}}\Big)\frac{f(u_{n})u_{n}}{|x|^{\alpha}}dx \rightarrow 0.
    \end{aligned}
\end{eqnarray*}
And then we conclude that $\|v_{n}\|\rightarrow 0$, repeating the same argument we get $\|u_{n}\|\rightarrow 0$. Hence, by case $1$, this situation cannot occur. Consequently, we obtain $(u_{0},v_{0})\neq 0$. So, $(u_{0},v_{0})$ is a nontrivial weak solution of $(\ref{a})$. Then choosing $\psi=u_{0}^{-}=\{-u_{0},0\}$ and $\varphi=0$ in $(\ref{PRCb})$, we obtain
\begin{eqnarray*}
    \begin{aligned}\displaystyle
    -\|u_{0}^{-}\|^{2}=\int_{\mathbb{R}^{2}}\Big(I_{\mu_{1}}\ast\frac{G(v_{0})}{|x|^{\alpha}}\Big)\frac{g(v_{0})u_{0}^{-}}{|x|^{\alpha}}\geq0,
    \end{aligned}
\end{eqnarray*}
which implies that $u_{0}^{-}=0$. Similarly, we deduce $v_{0}^{-}=0$, thus $u_{0},v_{0}\geq0$. If $u_{0}=0$, it is obvious to obtain $v_{0}=0$, allowing us to conclude that $u_{0},v_{0}>0$. Thus, the proof is complete.
\qed

\smallskip

\noindent{\bfseries Acknowledgements:}
The research has been supported by National Natural Science Foundation of China 11971392, Natural Science
Foundation of Chongqing, China cstc2021ycjh-bgzxm0115.


\begin{thebibliography}{99}

\bibitem{ack}
N. Ackermann, {\em On a periodic Schr$\ddot{o}$dinger equation with nonlocal superlinear part}, Math. Z. {\bf 248} (2004) 423-443.

\bibitem{AdY}
Adimurthi, Y. Yang, {\em An interpolation of Hardy inequality and Trudinger-Moser inequality in $\mathbb{R}^{N}$ and its applications}, Int. Math. Res. Not. {\bf13} (2010) 2394-2426.

\bibitem{Alb}
F. S. B. Albuquerque, J. M. do $\acute{O}$, E. S. Medeiros, {\em On a class of Hamiltonian systems involving unbounded or decaying potential in dimension two}, Math. Nachr. {\bf289} (2016) 1568-1584.

\bibitem{ACTT}
C. O. Alves, D. Cassani, C. Tarsi, M. B. Yang, {\em Existence and concentration of ground state solutions for a critical nonlocal Schr$\ddot{o}$dinger equation in $\mathbb{R}^{2}$}, J. Differ. Equ. {\bf261} (2016) 1933-1972.

\bibitem{BATTA}
L. Battaglia, J. Van Schaftingen, {\em Existence of groundstates for a class of nonlinear Choquard equations in the plane}, Adv. Nonlinear Stud. {\bf17} (2017) 581-594.

\bibitem{Ben}
A. Benaissa, B. Khaldi, {\em On a singular class of Hamiltonian systems in dimension two}, Z. Anal. Anwend. {\bf33} (2014) 199-215.

\bibitem{Bo}
D. Bonheure, E. M. dos Santos, H. Tavares, {\em Hamiltonian elliptic systems: a guide to variational frameworks}, Port. Math. {\bf71} (2014) 301-395.

\bibitem{Ca}
D. M. Cao, {\em Nontrivial solution of semilinear elliptic equation with critical exponent in $\mathbb{R}^{2}$}, Commun. Partial Differ. Equ. {\bf17} (1992) 407-435.

\bibitem{FdR}
D. G. de Figueiredo, J. M. do $\acute{O}$, B. Ruf, {\em Critical and subcritical elliptic systems in dimension two}, Indiana Univ. Math. J. {\bf 53(4)} (2004) 1037-1054.

\bibitem{FJZ}
D. G. de Figueiredo, J. M. do $\acute{O}$, J. J. Zhang, {\em Ground state solutions of Hamiltonian elliptic systems in dimension two},  Proc. R. Soc. Edinb. {\bf150} (2020) 1737-1768.

\bibitem{DDRU}
D. G. de Figueiredo, O. H. Miyagaki, B. Ruf, {\em Elliptic equations in $\mathbb{R}^{2}$ with nonlinearities in the critical growth range}, Calc. Var. Partial Differ. Equ. {\bf 3} (1995) 139-153.

\bibitem{FR}
H. Fr$\ddot{o}$hlich, {\em Theory of electrical breakdown in ionic crystal}, Proc. R. Soc. London, Ser. A {\bf 160(901)} (1937) 230-241.

\bibitem{LLU}
N. Lam, G. Lu, {\em Elliptic equations and systems with subcritical and critical exponential growth without the Ambrosetti-Rabinowitz condition}, J. Geom. Anal. {\bf24} (2014) 118-143.

\bibitem{Milan}
Y.R.S. Leuyacc, S. H. M. Soares, {\em On a Hamiltonian system with critical exponential growth}, Milan J. Math. {\bf87} (2019) 105-140.

\bibitem{MMA}
B. B. V. Maia, O. H. Miyagaki, {\em On a class of Hamiltonian Choquard-type elliptic systems}, J. Math. Phys. {\bf61} (2020) 011502.

\bibitem{MMB}
B. B. V. Maia, O. H. Miyagaki, {\em Existence and nonexistence results for a class of Hamiltonian Choquard-type elliptic systems with lower critical growth on $\mathbb{R}^{2}$}, Proc. R. Soc. Edinb. (2021) 1-28.

\bibitem{MVb}
V. Moroz, J. Van Schaftingen, {\em A guide to the Choquard equation}, J. Fixed Point Theory Appl. {\bf19} (2017) 773-813.

\bibitem{QT}
D. D. Qin, X. H. Tang, {\em On the planar Choquard equation with indefinite potential and critical exponential growth}, J. Differ. Equ. {\bf285} (2021) 40-98.

\bibitem{QTZA}
D. D. Qin, X. H. Tang, J. Zhang, {\em Ground states for planar Hamiltonian elliptic systems with critical exponential growth}. J. Differ. Equ. {\bf 308} (2022) 130-159.

\bibitem{Rabi}
P. H. Rabinowitz, {\em Minimax Methods in Critical Point Theory with Applications to Differential Equations}, CBMS Reg. Conf. Ser. Math., vol. {\bf65}, Amer. Math. Soc., Providence, RI, (1986).

\bibitem{Souza}
M. de Souza, {\em On a singular Hamiltonian elliptic systems involving critical growth in dimension two}. Commun. Pure Appl. Anal. {\bf11} (2012) 1859-1874.

\bibitem{SW}
E. M. Stein, G. Weiss, {\em Fractional integrals on n-dimensional Euclidean spaces}, J. Math. Mech. {\bf7} (1958) 503-514.
\end{thebibliography}
 \end{document}